\documentclass[12pt]{article}
\usepackage{amsmath}
\usepackage{amssymb}
\usepackage{amsthm}
\usepackage{verbatim}
\usepackage{mathrsfs}
\usepackage[left=2cm,right=2cm,bottom=2cm]{geometry}

\newcommand{\D}{{\mathcal D}}

\newcommand{\R}[0]{\mathbb R}

\newcommand{\Ds}[0]{\mathcal D}

\newtheorem{Th}{Theorem}[section]
\newtheorem{Lemma}{Lemma}[section]
\newtheorem{Prop}[Lemma]{Proposition}
\newtheorem{Coro}[Th]{Corollary}
\newtheorem{Def}{Definition}[section]
\newtheorem{Rem}{Remark}[section]

\begin{document}

\title{On a Lagrangian formulation of the incompressible Euler equation}
\author{
Hasan Inci\\
EPFL SB MATHAA PDE\\
MA C1 637 (B\^atiment MA)\\
Station 8\\
CH-1015 Lausanne, Switzerland\\
{\it email:} hasan.inci@epfl.ch
}

\maketitle

\begin{abstract}
In this paper we show that the incompressible Euler equation on the Sobolev space $H^s(\R^n)$, $s > n/2+1$, can be expressed in Lagrangian coordinates as a geodesic equation on an infinite dimensional manifold. Moreover the Christoffel map describing the geodesic equation is real analytic. The dynamics in Lagrangian coordinates is described on the group of volume preserving diffeomorphisms, which is an analytic submanifold of the whole diffeomorphism group. Furthermore it is shown that a Sobolev class vector field integrates to a curve on the diffeomorphism group.
\end{abstract}

\noindent
{\bf Keywords: Euler equation, Diffeomorphism group}

\section{Introduction}\label{section_introduction}

The initial value problem for the incompressible Euler equation in $\R^n$, $n \geq 2$, reads as:
\begin{eqnarray}
\nonumber
\partial_t u + (u \cdot \nabla) u &=& -\nabla p \\
\label{E}
\operatorname{div} u &=& 0 \\
\nonumber
u(0)&=& u_0
\end{eqnarray}
where $u(t,x)=\big(u_1(t,x),\ldots,u_n(t,x)\big)$ is the velocity of the fluid at time $t \in \R$ and position $x \in \R^n$, $u \cdot \nabla = \sum_{k=1}^n u_k \partial_k$ acts componentwise on $u$, $\nabla p$ is the gradient of the pressure $p(t,x)$, $\operatorname{div} u=\sum_{k=1}^n \partial_k u_k$ is the divergence of $u$ and $u_0$ is the value of $u$ at time $t=0$ (with assumption $\operatorname{div} u_0 = 0$). The system \eqref{E} (going back to Euler \cite{euler}) describes a fluid motion without friction. The first equation in \eqref{E} reflects the conservation of momentum. The second equation in \eqref{E} says that the fluid motion is incompressible, i.e. that the volume of any fluid portion remains constant during the flow.\\
The unknowns in \eqref{E} are $u$ and $p$. But as we will see later one can express $\nabla p$ in terms of $u$. Thus the evolution of system \eqref{E} is completely described by $u$. Therefore we will speak in the sequel of the solution $u$ instead of the solution $(u,p)$.\\
 Consider now a fluid motion determined by $u$. If one fixes a fluid particle which at time $t=0$ is located at $x \in \R^n$ and whose position at time $t \geq 0$ we denote by $\varphi(t,x) \in \R^n$, we get the following relation between $u$ and $\varphi$
\[
 \partial_t \varphi(t,x) = u\big(t,\varphi(t,x)),
\]
i.e. $\varphi$ is the flow-map of the vectorfield $u$. The second equation in \eqref{E} translates to the well-known relation $\det(d\varphi) \equiv 1$, where $d\varphi$ is the Jacobian of $\varphi$ -- see Majda, Bertozzi \cite{majda}. In this way we get a description of system \eqref{E} in terms of $\varphi$. The description of \eqref{E} in the $\varphi$-variable is called the Lagrangian description of \eqref{E}, whereas the description in the $u$-variable is called the Eulerian description of \eqref{E}. One advantage of the Lagrangian description of \eqref{E} is that it leads to an ODE formulation of \eqref{E}. This was already used in Lichtenstein \cite{lichtenstein} and Gunter \cite{gunter} to get local well-posedness of \eqref{E}.\\ \\
To state the result of this paper we have to introduce some notation. For $s \in \R_{\geq 0}$ we denote by $H^s(\R^n)$ the Hilbert space of real valued functions on $\R^n$ of Sobolev class $s$ and by $H^s(\R^n;\R^n)$ the vector fields on $\R^n$ of Sobolev class $s$ -- see Adams \cite{adams} or Inci, Topalov, Kappeler \cite{composition} for details on Sobolev spaces. We will often need the fact that for $n \geq 1$, $s > n/2$ and $0 \leq s' \leq s$ multiplication
\begin{equation}\label{multiplication}
 H^s(\R^n) \times H^{s'}(\R^n) \to H^{s'}(\R^n),\quad (f,g) \mapsto f \cdot g
\end{equation}
is a continuous bilinear map.\\
 The notion of solution for \eqref{E} we are interested in are solutions which lie in $C^0\big([0,T];H^s(\R^n;\R^n)\big)$ for some $T > 0$ and $s > n/2+1$. This is the space of continuous curves on $[0,T]$ with values in $H^s(\R^n;\R^n)$. To be precise we say that $u,\nabla p \in C^0\big([0,T];H^s(\R^n;\R^n)\big)$ is a solution to \eqref{E} if
\begin{equation}\label{RE}
 u(t) = u_0 + \int_0^t -(u(\tau) \cdot \nabla) u(\tau) - \nabla p(\tau) \;d\tau \quad \forall 0 \leq t \leq T
\end{equation}
and $\operatorname{div} u(t)=0$ for all $0 \leq t \leq T$ holds. As $s-1>n/2$ we know by the Banach algebra property of $H^{s-1}(\R^n)$ that the integrand in \eqref{RE} lies in $C^0\big([0,T];H^{s-1}(\R^n;\R^n)\big)$. Due to the Sobolev imbedding and the fact $s > n/2+1$ the solutions considered here are $C^1$ (in the $x$-variable slightly better than $C^1$) and are thus solutions for which the derivatives appearing in \eqref{E} are classical derivatives. \\
The discussion above shows that in this paper the state-space of \eqref{E} in the Eulerian description is $H^s(\R^n;\R^n)$, $s > n/2+1$. The state-space of \eqref{E} in the Lagrangian description is given by 
\[
 \Ds^s(\R^n) = \big\{ \varphi:\R^n \to \R^n \;\big|\; \varphi - \operatorname{id} \in H^s(\R^n;\R^n) \mbox{ and } \det d_x\varphi > 0, \;\forall x \in \R^n\big\}
\]
where $\operatorname{id}:\R^n \to \R^n$ is the identity map. Due to the Sobolev imbedding and the condition $s > n/2+1$ the space of maps $\Ds^s(\R^n)$ consists of $C^1$-diffeomorphisms -- see Palais \cite{palais} -- and can be identified via $\Ds^s(\R^n) - \operatorname{id} \subseteq H^s(\R^n;\R^n)$ with an open subset of $H^s(\R^n;\R^n)$. Thus $\Ds^s(\R^n)$ has naturally a real analytic differential structure (for real analyticity we refer to Whittlesey \cite{analyticity}) with the natural identification of the tangent space 
\[
 T\Ds^s(\R^n) \simeq \Ds^s(\R^n) \times H^s(\R^n;\R^n).
\] 
Moreover it is known that $\Ds^s(\R^n)$ is a topological group under composition and that for $0 \leq s' \leq s$ the composition map
\begin{equation}\label{composition}
 H^{s'}(\R^n) \times \Ds^s(\R^n) \to H^{s'}(\R^n), \quad (f,\varphi) \mapsto f \circ \varphi
\end{equation}
is continuous -- see Cantor \cite{cantor_thesis} and Inci, Topalov, Kappeler \cite{composition}. That $\Ds^s(\R^n)$ is the right choice as configuration space for \eqref{E} in Lagrangian coordinates is justified by the fact that every $u \in C^0\big([0,T];H^s(\R^n;\R^n)\big)$, $s > n/2+1$, integrates uniquely to a $\varphi \in C^1\big([0,T];\Ds^s(\R^n)\big)$ fullfilling
\[
 \partial_t \varphi(t) = u(t) \circ \varphi(t) \quad \mbox{for all } 0 \leq t \leq T 
\]
-- see Fischer, Marsden \cite{fischer} or Inci \cite{thesis} for an alternative proof. \\ \\
For the rest of this section we assume $n \geq 2$, $s > n/2+1$ and for $X,Y$ real Banach spaces we use the notation $L^2(X;Y)$ for the real Banach space of continuous bilinear maps from $X \times X$ to $Y$. With this we can state the main result of this paper:

\begin{Th}\label{th_main}
Let $n \geq 2$ and $s > n/2+1$. Then there is a real analytic map
\[
 \Gamma:\Ds^s(\R^n) \to L^2\big(H^s(\R^n;\R^n);H^s(\R^n;\R^n)\big), \quad \varphi \mapsto [(v,w) \mapsto \Gamma_\varphi(v,w)]
\]
called the Christoffel map for which the geodesic equation
\begin{equation}\label{geodesic_eq}
\partial_t^2 \varphi = \Gamma_\varphi(\partial_t \varphi,\partial_t \varphi);\quad \varphi(0)=\operatorname{id},\partial_t \varphi(0)=u_0 \in H^s(\R^n;\R^n)
\end{equation}
is a description of \eqref{E} in Lagrangian coordinates. More precisely, any $\varphi$ solving \eqref{geodesic_eq} on $[0,T]$, $T > 0$, with $\operatorname{div} u_0 = 0$ generates a solution to \eqref{E} by the formula $u:=\partial_t \varphi \circ \varphi^{-1}$ and on the other hand any $u$ solving \eqref{E} on $[0,T]$ integrates to a $\varphi$ solving \eqref{geodesic_eq} on $[0,T]$. 
\end{Th}

By ODE theory -- see Dieudonn\'e \cite{dieudonne} -- and the continuity of the composition map \eqref{composition} we immediately get the following corollary (this result, using a different method, goes back to Kato \cite{kato})

\begin{Coro}\label{coro_kato}
Let $n \geq 2$ and $s > n/2+1$. Then \eqref{E} is locally well-posed in $H^s(\R^n)$.
\end{Coro}

Connected to a geodesic equation like \eqref{geodesic_eq} is the notion of an exponential map -- see Lang \cite{lang}. The domain of definition for the exponential map is the set $U \subseteq H^s(\R^n;\R^n)$ consisting of initial values $u_0 \in H^s(\R^n;\R^n)$ for which the geodesic equation \eqref{geodesic_eq} has a solution on the interval $[0,1]$.  It turns out that $U$ is star-shaped with respect to $0$ and is an open neighborhood of $0$. With this we define 

\begin{Def}\label{def_exp}
The exponential map is defined as
\[
 \exp:U \to \Ds^s(\R^n),\quad u_0 \mapsto \varphi(1;u_0)
\]
where $\varphi(1;u_0)$ denotes the value of the solution $\varphi$ of \eqref{geodesic_eq} at time $t=1$ for the initial condition $\partial_t \varphi(0)=u_0$.
\end{Def}

By ODE theory we know that $exp$ is a real analytic map. Moreover we can describe every solution of \eqref{geodesic_eq} by considering the curves $t \mapsto \exp(t u_0)$ as is usual for geodesic equations. A further corollary of Theorem \ref{th_main} is

\begin{Coro}\label{coro_analytic_trajectory}
The trajectories of the fluid particles moving according to \eqref{E} are analytic.
\end{Coro}

\begin{proof}[Proof of Corollary \ref{coro_analytic_trajectory}]
Fix $x \in \R^n$ and define $\varphi(t):=\exp(t u_0)$. Then the trajectory of the fluid particle which starts at time $t=0$ at $x$ is given by $t \mapsto \varphi(t,x)$. By Theorem \ref{th_main} we know that
\[
 [0,T] \mapsto H^s(\R^n;\R^n),\quad t \mapsto \varphi(t)-\operatorname{id}
\]
is analytic. Here $T > 0$ is any time up to which the fluid motion exists for sure. As $s > n/2+1$ we know by the Sobolev imbedding that the evaluation map at $x \in \R^n$ 
\[
 H^s(\R^n) \to \R,\quad f \mapsto f(x)
\]
is a continuous linear map. Thus $t \mapsto \varphi(t,x)-x$ is analytic. Hence the claim.
\end{proof}

\emph{Related work}: To use an ODE-type approach for \eqref{E} via a Lagrangian formulation is already present in the works of Lichtenstein \cite{lichtenstein} and Gunter \cite{gunter}. One can also get analyticity in Lagrangian coordinates by using their successive approximation procedure.\\
The idea to express \eqref{E} as a geodesic equation on the ''Lie group'' $\Ds$, the group of diffeomorphisms, goes back to Arnold \cite{arnold}. In Ebin, Marsden \cite{ebin_marsden}, Ebin and Marsden worked out Arnold's idea by proving the analog of Theorem \ref{th_main} for the Sobolev spaces $H^s(M)$, where $M$ is a compact, smooth and oriented manifold of dimension $n$ and $s > n/2+1$, with the difference that they proved the Christoffel map $\Gamma$ to be smooth and not analytic (it is not so clear to us whether $\Gamma$ is analytic for all these $M$). Later Cantor \cite{cantor} showed the analog of Theorem \ref{th_main} for weighted Sobolev spaces on the whole space $H^s_w(\R^n)$, $s > n/2+1$ (Cantor stated it with $\Gamma$ smooth, but one can show that his $\Gamma$ is analytic). In Serfati \cite{serfati} the analog of Theorem \ref{th_main} was shown for $C^{k,\alpha}$-spaces over $\R^n$, $k \geq 1$ and $0 < \alpha < 1$. Most recently analytic dependence in the Lagrangian coordinates was shown to be true in the case of Sobolev spaces $H^s(\mathbb T^n)$, $s > n/2+1$, in Shnirelman \cite{shnirelman} and in the case of H\"older spaces $C^{1,\alpha}(\mathbb T^n)$, $0 < \alpha < 1$, in Frisch, Zheligovsky \cite{frisch} for fluid motion in the $n$-dimensional torus $\mathbb T^n=\R^n/\mathbb Z^n$.\\
As an application of the results of this paper we prove in Inci \cite{nonuniform} that the solution map of the incompressible Euler equation is nowhere locally Lipschitz and nowhere differentiable.\\ \\
This paper is more or less an excerpt from the thesis Inci \cite{thesis}. So omitted proofs or references where they can be found are given in Inci \cite{thesis}.

\section{Alternative Eulerian description}\label{section_alternative}

The goal of this section is to give an alternative formulation of \eqref{E} by replacing $\nabla p$ with an expression in $u$. For this we will use an idea of Chemin \cite{chemin}. Throughout this section we assume $n \geq 2$ and $s > n/2+1$.\\
To motivate the approach, we apply $\operatorname{div}$ to the first equation in \eqref{E} and use $\operatorname{div} u=0$ to get
\begin{equation}\label{laplace_p}
 -\Delta p = \sum_{j,k=1}^n \partial_j u_k \partial_k u_j = \sum_{j,k=1}^n \partial_j \partial_k (u_j u_k).
\end{equation}
In order to invert the Laplacian $\Delta$ we will use a cut-off in Fourier space. For this we denote by $\chi$ the characteristic function of the closed unit ball in $\R^n$, i.e. $\chi(\xi)=1$ for $|\xi| \leq 1$ and $\chi(\xi)=0$ otherwise. The continuous linear operator $\chi(D)$ on $L^2(\R^n):=L^2_\R(\R^n)$ is defined by
\[
 \chi(D):L^2(\R^n) \to L^2(\R^n),\quad f \mapsto \mathcal F^{-1}\left[\chi(\xi)\mathcal F[f](\xi) \right]
\]
where $\mathcal F$ is the Fourier transform and $\mathcal F^{-1}$ its inverse. We define the Fourier transform of $g \in L^1(\R^n)$ as the following complex-valued function $\mathcal F[g]:\R^n \to \mathbb C$ (with the usual extension to $L^2(\R^n)$)
\[
 \mathcal F[g](\xi) := \frac{1}{(2\pi)^{n/2}} \int_{\R^n} e^{-ix \cdot \xi} g(x) \;dx,\quad \xi \in \R^n
\]
where $x \cdot \xi=x_1 \xi_1 + \ldots + x_n \xi_n$ is the Euclidean inner product in $\R^n$. We have for $s_1,s_2 \geq 0$
\begin{equation}\label{chi_smoothing}
 ||\chi(D)f||_{s_1+s_2} \leq 2^{s_2/2} ||f||_{s_1},\quad \forall f \in H^{s_1}(\R^n)
\end{equation}
where $||g||_{s'}:=||\,(1+|\xi|^{s'/2}) \, |\mathcal F[g](\xi)| \;||_{L^2}$ for $g \in H^{s'}(\R^n)$, $s' \geq 0$. We use \eqref{laplace_p} to rewrite $-\nabla p$
\[
 -\nabla p = \nabla \left(\Delta^{-1} \big(1-\chi(D)\big) \sum_{j,k=1}^n \partial_j u_k \partial_k u_j + \Delta^{-1} \chi(D) \sum_{j,k=1}^n \partial_j \partial_k (u_j u_k)\right).
\]
Using this expression we replace \eqref{E} by
\begin{equation}\label{alternative_E}
 \partial_t u + (u \cdot \nabla) u = \nabla B(u,u),\quad u(0)=u_0 \in H^s(\R^n;\R^n)
\end{equation}
where $B(v,w)=B_1(v,w)+B_2(v,w)$ for $v,w \in H^s(\R^n;\R^n)$ with
\[
 B_1(v,w)=\Delta^{-1} \big(1-\chi(D)\big) \sum_{j,k=1}^n \partial_j v_k \partial_k w_j
\]
and
\[
 B_2(v,w)=\Delta^{-1} \chi(D) \sum_{j,k=1}^n \partial_j \partial_k(v_j w_k).
\]
As $\Delta^{-1}\big(1-\chi(D)\big):H^{s-1}(\R^n) \to H^{s+1}(\R^n)$ is a continuous linear map we get by the Banach algebra property of $H^{s-1}(\R^n)$ that
\begin{eqnarray*}
 B_1:H^s(\R^n;\R^n) \times H^s(\R^n;\R^n) &\to& H^{s+1}(\R^n)\\
(v,w) &\mapsto& \Delta^{-1}\big(1-\chi(D)\big) \sum_{j,k=1}^n \partial_j v_k \partial_k w_j
\end{eqnarray*}
is a continuous bilinear map. And as $\Delta^{-1} \chi(D) \partial_j \partial_k:H^s(\R^n) \to H^{s+1}(\R^n)$ is a continuous linear map for any $1 \leq j,k \leq n$ we get by the Banach algebra property of $H^s(\R^n)$ that
\begin{eqnarray*}
 B_2:H^s(\R^n;\R^n) \times H^s(\R^n;\R^n) &\to& H^{s+1}(\R^n)\\
(v,w) &\mapsto& \Delta^{-1} \chi(D) \sum_{j,k=1}^n \partial_j \partial_k (v_j w_k)
\end{eqnarray*}
is a continuous bilinear map. Altogether we see that
\[
 \nabla B:H^s(\R^n;\R^n) \times H^s(\R^n;\R^n) \to H^s(\R^n;\R^n)
\]
is a continuous bilinear map. Equation \eqref{alternative_E} is to be understood in the sense that $u$ is a solution to \eqref{alternative_E} on $[0,T]$ for some $T>0$ if $u \in C^0\big([0,T];H^s(\R^n;\R^n)\big)$ with
\begin{equation}\label{alternative_RE}
 u(t) = u_0 + \int_0^t \nabla B\big(u(\tau),u(\tau)\big) - (u(\tau) \cdot \nabla) u(\tau) \;d\tau 
\end{equation}
for any $0 \leq t \leq T$. By the Banach algebra property of $H^{s-1}(\R^n)$ the integrand in \eqref{alternative_RE} lies in $C^0\big([0,T];H^{s-1}(\R^n;\R^n)\big)$. \\
To consider \eqref{alternative_E} instead of \eqref{E} is justified by the following proposition 

To consider \eqref{alternative_E} instead of \eqref{E} is justified by Proposition \ref{prop_alternative}. Proposition \ref{prop_alternative} shows in particular that for solutions of \eqref{alternative_E} the condition $\operatorname{div} u(t)=0$ is preserved if it is true for $t=0$.

\section{Proof of Theorem \ref{th_main}}\label{section_proof}

The goal of this section is to prove Theorem \ref{th_main}. To do that we will formulate the alternative equation \eqref{alternative_E} in Lagrangian coordinates. As usual we assume $n \geq 2$ and $s > n/2+1$. To motivate the approach consider $u$ solving \eqref{alternative_E} and $\varphi$ its flow, i.e. $\varphi$ is determined by the relation $\partial_t \varphi = u \circ \varphi$. Taking the $t$-derivative in this relation we get
\[
 \partial_t^2 \varphi = \big(\partial_t u + (u \cdot \nabla) u \big) \circ \varphi = \nabla B(u,u) \circ \varphi. 
\]
Replacing $u$ by $u=\partial_t \varphi \circ \varphi^{-1}$ we get
\[
 \partial_t^2 \varphi = \nabla B(\partial_t \varphi \circ \varphi^{-1},\partial_t \varphi \circ \varphi^{-1}) \circ \varphi.
\]
So our candidate for the $\Gamma$ in Theorem \ref{th_main} is 
\begin{equation}\label{def_gamma}
 \Gamma_\varphi(v,w) := \nabla B(v \circ \varphi^{-1},w \circ \varphi^{-1}) \circ \varphi. 
\end{equation}
The key ingredient for the proof of Theorem \ref{th_main} is the following proposition

\begin{Prop}\label{prop_analytic_gamma}
The map
\[
 \Gamma:\Ds^s(\R^n) \to L^2\big(H^s(\R^n;\R^n);H^s(\R^n;\R^n)\big),\quad \varphi \mapsto [ (v,w) \mapsto \Gamma_\varphi(v,w)]
\]
with $\Gamma_\varphi(v,w)$ as in \eqref{def_gamma} is real analytic.
\end{Prop}

Before we proof Proposition \ref{prop_analytic_gamma} we have to make some preparation. We introduce the following subspace of $L^2(\R^n)$
\[
 H^\infty_\Xi(\R^n) := \big\{ g \in L^2(\R^n) \;\big|\; \operatorname{supp} \mathcal F[g] \subseteq \Xi \big\}
\]
where $\Xi \subseteq \R^n$ is the closed unit ball and $\operatorname{supp} f$ denotes the support of a function $f$. The space $H^\infty_\Xi(\R^n)$ is a closed subspace of $L^2(\R^n)$, lies in $\cap_{s' \geq 0} H^{s'}(\R^n)$ and consists of entire functions (i.e. analytic functions on $\R^n$ with convergence radius $R=\infty$). Note that $\chi(D)$ maps $H^{s'}(\R^n)$, $s' \geq 0$, into $H^\infty_\Xi(\R^n)$. In the sequel we will also use the vector-valued analog $H^\infty_\Xi(\R^n;\R^n)=\{(f_1,\ldots,f_n) | f_k \in H^\infty_\Xi(\R^n), \; \forall 1 \leq k \leq n\}$. The space $H^\infty_\Xi$ has good properties with regard to the composition map (in contrast to its bad behaviour in the $H^s$ space -- see Inci \cite{thesis}):\\
Denoting  by $L(X;Y)$, $X,Y$ real Banach spaces, the real Banach space of continuous linear maps from $X$ to $Y$ we have

\begin{Lemma}\label{lemma_composition1}
Let $n \geq 2$ and $s > n/2+1$. Then 
\[
 \Ds^s(\R^n) \to L\big(H^\infty_\Xi(\R^n);H^s(\R^n)\big),\quad \varphi \mapsto [f \mapsto f \circ \varphi]
\]
is real analytic.
\end{Lemma}

Recall that the differential structure of $\Ds^s(\R^n)$ is given by identifying it with the open set $\Ds^s(\R^n) - \mbox{id} \subseteq H^s(\R^n;\R^n)$.

\begin{proof}
 Let $f \in H_\Xi^\infty(\R^n)$. We know that $f$ is an entire function and hence admits a power series expansion for any $x,y \in \R^n$
\[
 f(x+y)=\sum_{|\alpha| \geq 0} \frac{1}{\alpha !} \partial^\alpha f(x) y^\alpha
\]
where we use the multi-index notation, i.e. for a multi-index $\alpha=(\alpha_1,\ldots,\alpha_n) \in \mathbb Z^n_{\geq 0}$, 
\[
 \partial^\alpha f(x) = \partial_1^{\alpha_1} \cdots \partial_n^{\alpha_n} f(x),\quad \alpha!=\alpha_1! \cdots \alpha_n !
\]
and for $y=(y_1,\ldots,y_n) \in \R^n$, $y^\alpha= y_1^{\alpha_1} \cdots y_n^{\alpha_n}$. For the derivative $\partial^\alpha f$ we have from \eqref{chi_smoothing}
\[ 
 ||\partial^\alpha f||_s \leq ||f||_{s+|\alpha|} \leq 2^{(s+|\alpha|)/2} ||f||_{L^2}.
\]
Writing $\varphi = \operatorname{id} + g$, $g=(g_1,\ldots,g_n) \in H^s(\R^n;\R^n)$ we have with the notation $g^\alpha(x)=g_1^{\alpha_1}(x) \cdots g_n^{\alpha_n}(x)$, pointwise for all $x \in \R^n$
\[
 f\big(\varphi(x)\big) = f\big(x+g(x)\big) = \sum_{|\alpha| \geq 0} \frac{1}{\alpha!} \partial^\alpha f(x) g^\alpha(x)
\]
or formally as an identity in $\mathcal L\big(H_\Xi^\infty(\R^n),H^s(\R^n)\big)$
\begin{equation}\label{power_series}
 f \mapsto f \circ \varphi \equiv f \mapsto \sum_{k \geq 0} Q_k(g)(f)
\end{equation}
where $Q_k(g)$ is a linear differential operator of order $k$ whose coefficients are homogeneous polynomials in the components of $g \in H^s(\R^n;\R^n)$, $Q_k(g)=\sum_{|\alpha|=k} \frac{1}{\alpha!} g^\alpha \partial^\alpha$, acting on functions $f \in H_\Xi^\infty(\R^n)$ as
\[
 Q_k(g)(f) = \sum_{|\alpha|=k} \frac{1}{\alpha!} g^\alpha \partial^\alpha f.
\]
Note that $Q_k(g):H_\Xi^\infty(\R^n) \to H^s(\R^n)$ is a bounded linear map. Indeed we have by the Banach algebra property of $H^s(\R^n)$ for any multi-index $\alpha$ with $|\alpha|=k$
\[
 ||g^\alpha \partial^\alpha f||_s \leq C^k ||g||_s^k ||\partial^\alpha f||_s \leq C^k ||g||_s^k ||f||_{s+k}
\]
and hence 
\[
 ||Q_k(g)(f)||_s \leq \left(\sum_{|\alpha|=k} \frac{1}{\alpha!}\right) C^k 2^{(s+k)/2} ||f||_{L^2} ||g||_s^k  = \frac{1}{k!} n^k C^k 2^{(2+k)/2} ||f||_{L^2} ||g||_s^k.
\]
Here we used that by the multinomial theorem,
\[
 \sum_{|\alpha|=k} \frac{k!}{\alpha!} = (1+\ldots+1)^k = n^k.
\]
Recall the norm of the operator $Q_k$ -- see Appendix \ref{appendix_analyticity} \eqref{radius_condition}
\[
 ||Q_k|| = \sup_{\mbox{\scriptsize $\begin{array}{c} ||g||_s \leq 1 \\ f \in H_\Xi^\infty(\R^n)\\ ||f||_{L^2} \leq 1 \end{array}$}} ||Q_k(g)(f)||_s.
\]
Altogether we have proved that $||Q_k|| \leq \frac{1}{k!} n^k C^k 2^{(2+k)/2}$ leading to 
\[
 \sup_{k \geq 0} ||Q_k|| r^k < \infty
\]
for all $r > 0$. Therefore the series \eqref{power_series} has convergence radius $R=\infty$. Now the pointwise limit $f \circ \varphi$ and the $H^s$-limit $\sum_{k \geq 0} Q_k(g)(f)$ must coincide. Thus we see that
\[
 \Phi:\Ds^s(\R^n) \to \mathcal L\big(H_\Xi^\infty(\R^n),H^s(\R^n)\big),\quad \varphi \mapsto \big[f \mapsto f \circ \varphi = \left(\sum_{k \geq 0} Q_k(g)\right) (f)\big]
\]
is real analytic. Again we identify here $\varphi$ with $g=\varphi - \operatorname{id}$. 

\end{proof}

\begin{Lemma}\label{lemma_composition2}
Let $n \geq 2$, $s > n/2+1$ and $0 \leq s' \leq s$. Then
\[
 \Ds^s(\R^n) \to L\big(H^{s'}(\R^n);H^\infty_\Xi(\R^n)\big),\quad \varphi \mapsto [f \mapsto \chi(D)(f \circ \varphi^{-1})]
\]
is real analytic.
\end{Lemma}

\begin{proof}
 First we consider
\begin{equation}\label{joint_composition}
 H^{s'}(\R^n) \times \Ds^s(\R^n) \to H_\Xi^\infty(\R^n),\quad (f,\varphi) \mapsto \chi(D)(f \circ \varphi^{-1})
\end{equation}
and show that it is weakly analytic in the sense that
\begin{equation}\label{weakly_analytic}
(f,\varphi) \mapsto \langle \chi(D) (f \circ \varphi^{-1}),g \rangle_{L^2}
\end{equation}
is real analytic for any fixed $g \in H_{\Xi}^\infty(\R^n)$. So choose $g \in H_\Xi^\infty(\R^n)$. Note that $\chi(D) g=g$. As
\[
 \langle \chi(D)(f\circ \varphi^{-1}),g \rangle_{L^2} = \langle f \circ \varphi^{-1}, \chi(D) g \rangle_{L^2}
\]
it then follows after a change of variable of integration $y=\varphi^{-1}(x)$ that
\begin{equation}\label{weak_integral}
 \int_{\R^n} f\circ \varphi^{-1} \cdot g \,dx = \int_{\R^n}f \cdot g \circ \varphi \cdot \det(d_y \varphi)\, dy.
\end{equation}
By Lemma \ref{lemma_composition1} it follows that
\[
 \Ds^s(\R^n) \to H^{s'}(\R^n),\quad \varphi \mapsto g \circ \varphi
\]
is real analytic with convergence radius $R=\infty$. In addition $\Ds^s(\R^n) \to H^{s-1}(\R^n)$, $\varphi \mapsto \det(d_x \varphi) -1$ is also real analytic with radius of convergence $R=\infty$, since it is a polynomial. Altogether one then concludes that the expression on the right-hand side of \eqref{weak_integral} is real analytic in $(f,\varphi)$ with radius of convergence $R=\infty$. As $g$ was arbitrary, we conclude from Proposition \ref{prop_weak_analytic} that \eqref{joint_composition} is real analytic with radius of convergence $R=\infty$. As the map \eqref{joint_composition} is linear in $f$ the claim of the Lemma holds -- for details see Appendix \ref{appendix_analyticity}
\end{proof}

We split the proof of Proposition \ref{prop_analytic_gamma} according to $B=B_1+B_2$ into two lemmas. In the sequel we will use the notation $R_\varphi$ for the right-composition, i.e. $R_\varphi f:=f \circ \varphi$. Note that $R_\varphi^{-1}=R_{\varphi^{-1}}$.

\begin{Lemma}\label{lemma_b1}
Let $n \geq 2$ and $s > n/2+1$. Then
\begin{eqnarray*}
 \Ds^s(\R^n) &\to& L^2\big(H^s(\R^n;\R^n);H^s(\R^n;\R^n)\big)\\
\varphi &\mapsto& [(v,w) \mapsto \nabla B_1(v \circ \varphi^{-1},w \circ \varphi^{-1}) \circ \varphi] 
\end{eqnarray*}
is real analytic.
\end{Lemma}

\begin{proof}[Proof of Lemma \ref{lemma_b1}]
Recall that $\nabla B_1(v,w)$ is given by
\[
 \nabla B_1(v,w)= \nabla \left( \Delta^{-1} \big(1- \chi(D)\big) \sum_{j,k=1}^n \partial_j v_k \partial_k w_j \right).
\]
It will be convenient to write $\Delta^{-1} \big(1-\chi(D)\big)$ as
\begin{equation}\label{convenient}
 \Delta^{-1} \big(1-\chi(D)\big) = \left(\chi(D) + \Delta \big(1-\chi(D)\big)\right)^{-1} - \chi(D).
\end{equation}
In a first step we will prove that for $A:=\chi(D) + \Delta \big(1-\chi(D)\big)$
\[
\Ds^s(\R^n) \to L\big(H^s(\R^n);H^{s-2}(\R^n)\big),\quad \varphi \mapsto [f \mapsto R_\varphi  A R_\varphi^{-1} f]
\]
is real analytic. From Lemma \ref{lemma_composition1} and \ref{lemma_composition2} we know that
\[
 \Ds^s(\R^n) \to L\big(H^s(\R^n);H^s(\R^n)\big),\quad \varphi \mapsto [f \mapsto R_\varphi \chi(D) R_\varphi^{-1} f]
\]
is real analytic. The same is of course true if we replace above $\chi(D)$ by $1-\chi(D)$. To proceed we prove that for any $1 \leq s' \leq s$ and $1 \leq k \leq n$
\begin{equation}\label{conjugation}
 \Ds^s(\R^n) \to L\big(H^{s'}(\R^n);H^{s'-1}(\R^n)\big),\quad \varphi \mapsto [f \mapsto R_\varphi \partial_k R_\varphi^{-1} f]
\end{equation}
is real analytic. We clearly have
\[
 R_\varphi \partial_k R_\varphi^{-1} f = \sum_{j=1}^n \partial_j f C_{jk}  
\]
where $(C_{jk})_{1 \leq j,k \leq n} = [d\varphi]^{-1}$, i.e the inverse matrix of the jacobian of $\varphi$. Note that the entries of $(C_{jk})_{1 \leq j,k \leq n}$ are polynomial expressions of the entries of $[d\varphi]$ divided by $\det(d\varphi)$. As $H^{s-1}$ is a Banach algebra and division by $\det(d\varphi)$ an analytic operation -- see Inci \cite{thesis} -- we get by \eqref{multiplication} that $\varphi \mapsto R_\varphi \partial_k R_\varphi^{-1} $ is real analytic as claimed. Writing
\[
 R_\varphi \Delta R_\varphi^{-1} = \sum_{k=1}^n R_\varphi \partial_k R_\varphi^{-1} R_\varphi \partial_k R_\varphi^{-1} 
\]
we thus see that $\varphi \mapsto R_\varphi \Delta R_\varphi^{-1}$ is also real analytic. Finally writing 
\begin{equation}\label{analytic_A}
 R_\varphi A R_\varphi^{-1} = R_\varphi \chi(D) R_\varphi^{-1} + R_\varphi \Delta R_\varphi^{-1} R_\varphi \big(1-\chi(D)\big) R_\varphi^{-1}
\end{equation}
we get that $\varphi \mapsto R_\varphi A R_\varphi^{-1}$ is real analytic. Denoting by $X,Y$ real Banach spaces and by $GL(X;Y) \subseteq L(X;Y)$ the open subset of invertible continuous linear operators from $X$ to $Y$ we know by the Neumann series -- see Dieudonn\'e \cite{dieudonne} -- that
\[
 \operatorname{inv}:GL(X;Y) \to GL(Y;X),\quad T \mapsto T^{-1}
\]
is real analytic. Therefore we get from the analyticity of \eqref{analytic_A} that
\[
 \Ds^s(\R^n) \to L\big(H^{s-2}(\R^n);H^s(\R^n)\big),\quad \varphi \mapsto R_\varphi A^{-1} R_\varphi^{-1} = \left(R_\varphi A R_\varphi^{-1}\right)^{-1}
\]
is real analytic. This implies by \eqref{convenient} that
\[
 \Ds^s(\R^n) \to L\big(H^{s-2}(\R^n);H^s(\R^n)\big),\quad \varphi \mapsto R_\varphi \Delta^{-1} \big(1-\chi(D)\big) R_\varphi^{-1} 
\]
is real analytic. By letting $\Delta^{-1}\big(1-\chi(D)\big)$ act componentwise we write
\begin{multline*}
 \nabla B_1(v \circ \varphi^{-1},w \circ \varphi^{-1}) \circ \varphi = \\
\big(R_\varphi \Delta^{-1} \big(1-\chi(D)\big) R_\varphi^{-1}\big) \big( R_\varphi \nabla R_\varphi^{-1} \big) \sum_{j,k=1}^n \big(R_\varphi \partial_j R_\varphi^{-1} v_k \big) \big(R_\varphi \partial_k R_\varphi^{-1} w_j\big) 
\end{multline*}
and we get from the considerations above that
\begin{eqnarray*}
 \Ds^s(\R^n) &\to& L^2\big(H^s(\R^n;\R^n);H^s(\R^n;\R^n)\big)\\
\varphi &\mapsto& [(v,w) \mapsto \nabla B_1(v \circ \varphi^{-1},w \circ \varphi^{-1}) \circ \varphi]
\end{eqnarray*}
is real analytic.
\end{proof}

\begin{Lemma}\label{lemma_b2}
Let $n \geq 2$ and $s > n/2+1$. Then
\begin{eqnarray*}
 \Ds^s(\R^n) &\to& L^2\big(H^s(\R^n;\R^n);H^s(\R^n;\R^n)\big)\\
\varphi &\mapsto& [(v,w) \mapsto \nabla B_2(v \circ \varphi^{-1},w \circ \varphi^{-1}) \circ \varphi] 
\end{eqnarray*}
is real analytic.
\end{Lemma}
\begin{proof}[Proof of Lemma \ref{lemma_b2}]
We write
\begin{equation}\label{b2_expression}
 \nabla B_2(v \circ \varphi^{-1},w \circ \varphi^{-1}) \circ \varphi = \sum_{j,k=1}^n R_\varphi \nabla  \Delta^{-1} \partial_j \partial_k \chi(D) R_\varphi^{-1}(v_j w_k). 
\end{equation}
By Lemma \ref{lemma_composition2} we know that $\varphi \mapsto \chi(D) R_\varphi^{-1}$ is real analytic with values in $L\big(H^s(\R^n);H^\infty_\Xi(\R^n)\big)$. Moreover for any $1 \leq j,k \leq n$
\[
 \nabla \Delta^{-1} \partial_j \partial_k:H^\infty_\Xi(\R^n) \to H^\infty_\Xi(\R^n;\R^n)
\]
is a continuous linear map. By Lemma \ref{lemma_composition1} we then see that the expression \eqref{b2_expression} is real analytic in $\varphi$ showing the claim. 
\end{proof}

\begin{proof}[Proof of Proposition \ref{prop_analytic_gamma}]
As $B=B_1+B_2$ the proof follows from Lemma \ref{lemma_b1} and Lemma \ref{lemma_b2}.
\end{proof}

Now we can prove the main theorem.

\begin{proof}[Proof of Theorem \ref{th_main}]
The analyticity statement for $\Gamma$ follows from Proposition \ref{prop_analytic_gamma}. To prove the first part of the second statement consider $\varphi \in C^2\big([0,T];\Ds^s(\R^n)\big)$, $T > 0$, solving
\begin{equation}\label{geodesic_eq2}
 \partial_t^2 \varphi= \Gamma_\varphi(\partial_t \varphi,\partial_t \varphi),\quad \varphi(0)=\operatorname{id},\partial_t \varphi(0)=u_0 \in H^s(\R^n;\R^n).
\end{equation}
We define $u:=\partial_t \varphi \circ \varphi^{-1}$. By the continuity of the group operations in $\Ds^s(\R^n)$ and by \eqref{composition} we know that $u \in C^0\big([0,T];H^s(\R^n;\R^n)\big)$. By the Sobolev imbedding we have $\varphi,\partial_t \varphi \in C^1\big([0,T]\times \R^n;\R^n)$. By the inverse function theorem we also have $\varphi^{-1} \in C^1\big([0,T] \times \R^n;\R^n)$. Hence $u \in C^1\big([0,T]\times \R^n;\R^n)$. Taking the pointwise $t$-derivative in the relation $\partial_t \varphi(t,x) = u(t,\varphi(t,x))$ leads to
\begin{equation}\label{t_derivative} 
\partial_t^2 \varphi = (\partial_t u + (u \cdot \nabla)u) \circ \varphi.
\end{equation}
Using the expression \eqref{def_gamma} corresponding to $\Gamma_\varphi(\partial_t \varphi,\partial_t \varphi)$ and using $u=\partial_t \varphi \circ \varphi^{-1}$ we get pointwise (for any $(t,x) \in [0,T] \times \R^n$ without writing the argument explicitly)
\[
 B(u,u) \circ \varphi = (\partial_t u + (u \cdot \nabla) u) \circ \varphi.
\]
Skipping the composition by $\varphi$ on both sides, we get by the fundamental lemma of calculus for any $(t,x) \in [0,T] \times \R^n$ (without writing the $x$-argument)
\begin{equation}\label{fundamental_lemma}
 u(t)=u_0 + \int_0^t B\big(u(\tau),u(\tau)\big) - \big(u(\tau) \cdot \nabla \big) u(\tau) \;d\tau. 
\end{equation}
The integrand in \eqref{fundamental_lemma} lies in $C^0\big([0,T];H^{s-1}(\R^n;\R^n)\big)$ so that \eqref{fundamental_lemma} is actually an identity in $H^{s-1}$, which shows that $u$ is a solution to the alternative formulation \eqref{alternative_E}.\\ Now it remains to prove the other direction. We take $u$ solving the alternative formulation \eqref{alternative_E}. We know that there is a unique $\varphi \in C^1\big([0,T];\Ds^s(\R^n)\big)$ solving
\[
 \partial_t \varphi = u \circ \varphi,\quad \varphi(0)=\operatorname{id}.
\]
The claim is that $\varphi$ solves the geodesic equation \eqref{geodesic_eq2}. First note that by the fact that $u$ is a solution to the alternative formulation \eqref{alternative_E} and by the Sobolev imbedding we have $u,\varphi \in C^1([0,T] \times \R^n;\R^n)$. Thus we also have $\partial_t \varphi \in C^1([0,T] \times \R^n;\R^n)$. Taking the $t$-derivative in $\partial_t \varphi=u \circ \varphi$ we get the same expression as in \eqref{t_derivative}. Using that $u$ is a solution to the alternative formulation \eqref{alternative_E} we get by the fundamental lemma of calculus pointwise for any $(t,x) \in [0,T] \times \R^n$ (dropping the $x$-argument)
\begin{eqnarray*}
 \partial_t \varphi(t) &=& u_0 + \int_0^t B\big(\partial_t \varphi(\tau) \circ \varphi(\tau)^{-1},\partial_t \varphi(\tau) \circ \varphi(\tau)^{-1}\big) \circ \varphi(\tau) \;d\tau\\
 &=& u_0 + \int_0^t \Gamma_{\varphi(\tau)}\big(\partial_t \varphi(\tau),\partial_t \varphi(\tau)\big) \;d\tau.
\end{eqnarray*}
But as the integrand is a continuous curve in $H^s(\R^n;\R^n)$ we see that $t \mapsto \varphi(t)$ solves the geodesic equation \eqref{geodesic_eq2}. This completes the proof.
\end{proof}

In view of the condition $\operatorname{div} u=0$, the state space of \eqref{E} in Lagrangian coordinates is actually $\Ds^s_\mu(\R^n) \subseteq \Ds^s(\R^n)$, the subgroup of volume-preserving diffeomorphisms, i.e.
\[
 \Ds^s_\mu(\R^n) := \big\{ \varphi \in \Ds^s(\R^n) \;\big| \; \det(d\varphi) \equiv 1 \big\}.
\]

One has -- see section \ref{section_submanifold} for the proof

\begin{Th}\label{th_submanifold}
Let $n \geq 2$ and $s > n/2+1$. Then $\Ds^s_\mu(\R^n)$ is a closed real analytic submanifold of $\D^s(\R^n)$.
\end{Th}

So the dynamics of \eqref{E} in Lagrangian coordinates is real analytic on $\Ds^s_\mu(\R^n)$ or expressed with the exponential map

\begin{Coro}
Let $n \geq 2$ and $s > n/2+1$. Then
\[
 \exp:U \cap H^s_\sigma(\R^n;\R^n) \to \Ds^s_\mu(\R^n)
\]
is real analytic.
\end{Coro}

\section{Integration of $H^s$-vector fields}\label{integration}

The goal of this section is to prove that we can integrate a $H^s$-vector field to a flow in $\Ds^s(\R^n)$. More precisely

\begin{Prop}\label{prop_integration}
Let $s > n/2+1$ and $T>0$. For a given $u \in C\big([0,T];H^s(\R^n;\R^n)\big)$ there is a unique $\varphi \in C^1\big([0,T];\Ds^s(\R^n)\big)$ solving
\[
 \partial_t \varphi = u \circ \varphi \mbox{ on } [0,T];\quad \varphi(0)=\operatorname{id} \in \Ds^s(\R^n).
\]
\end{Prop}

\begin{Rem}
Proposition \ref{prop_integration} was proved in \cite{fischer}. The idea there is the following. If we write $\varphi^{-1}=\operatorname{id}+f$, where $f \in C^0\big([0,T];H^s(\R^n;\R^n)\big)$, we get by differentiating $\varphi^{-1} \circ \varphi = \operatorname{id}$
\[
 \partial_t f \circ \varphi + (I_n + df) \circ \varphi \cdot \partial_t \varphi = 0 
\]
or the following transport equation for $f$
\[
 \partial_t f + u + df \cdot u=0.
\]
with coefficients in $H^s$. Now one can use the theory for linear symmetric hyperbolic systems developed in \cite{fischer} to solve this problem. But we will give a more ''dynamical systems''-proof.
\end{Rem}

The uniqueness part of the proposition is an easy task. Indeed by the Sobolev imbedding we see that $u$ is a uniformly Lipschitz vector field $u:[0,T] \times \R^n \to \R^n$ with respect to the spatial variable, because we have
\[
 |u(t,x)-u(t,y)| \leq C ||u(t)||_s |x-y| \leq CM |x-y|
\]
where $M=\max_{0 \leq t \leq T} ||u(t)||_s$. Thus we have a unique flow $\tilde \varphi:[0,T] \times \R^n \to \R^n$.\\ \\
\noindent
Before proving the proposition we will make some preparation. Since the composition map is linear in the first entry we can get the following local linear growth estimate.

\begin{Lemma}\label{lemma_linear_growth}
Let $s > n/2+1$, $0 \leq s' \leq s$ and $\varphi_\bullet \in \Ds^s(\R^n)$ be given. Then there is a neighborhood $\mathcal G$ of $\varphi_\bullet$ in $\Ds^s(\R^n)$ and a $C > 0$ with
\[
 \frac{1}{C} ||f||_{s'} \leq ||f \circ \varphi||_{s'} \leq C ||f||_{s'}
\]
for all $f \in H^{s'}(\R^n)$ and $\varphi \in \mathcal G$.
\end{Lemma}

\begin{proof}
Consider the composition map
\[
 \mu:H^{s'}(\R^n) \times \Ds^s(\R^n) \to H^{s'}(\R^n),\quad (f,\varphi) \mapsto f \circ \varphi
\]
which by \cite{composition} is continuous. As we have $\mu(0,\varphi_\bullet)=0$ there exist, by the continuity of $\mu$, $R > 0$ and a neighborhood $\mathcal G$ of $\varphi_\bullet$ such that we have
\[
 ||f \circ \varphi||_{s'} \leq 1
\]
for all $f \in H^{s'}(\R^n)$ with $||f||_{s'} \leq R$ and for all $\varphi \in \mathcal G$. By linearity we thus get
\[
 ||f \circ \varphi||_{s'} \leq \frac{1}{R} ||f||_{s'}
\]
for all $f \in H^{s'}(\R^n)$ and for all $\varphi \in \mathcal G$. The same reasoning gives, by shrinking $R$ and $\mathcal G$ if necessary,
\[
 ||g \circ \varphi^{-1}||_{s'} \leq \frac{1}{R} ||g||_{s'}
\]
for all $g \in H^{s'}(\R^n;\R^n)$ and $\varphi \in \mathcal G$. Replacing $g$ by $f \circ \varphi$ we get the claim.
\end{proof}

\noindent
The following lemma is a special case of Proposition \ref{prop_integration}. The proof follows the one given in \cite{ebin_marsden}.

\begin{Lemma}\label{lemma_integration}
Assume $s > n/2+2$. Then the claim of Proposition \ref{prop_integration} holds.
\end{Lemma}

\begin{proof}
Note that for $s>n/2+2$ the space $\Ds^{s-1}(\R^n)$ is defined. In a neighborhood of $\operatorname{id} \in \Ds^{s-1}(\R^n)$, let's say
\[
 \mathcal G^{s-1}_\varepsilon:=\big\{ \varphi \in \Ds^{s-1}(\R^n) \; \big| \; ||\varphi - \operatorname{id}||_{s-1} < \varepsilon \big\}
\]
we have by Lemma \ref{lemma_linear_growth} for some constant $C > 0$ 
\[
 || f \circ \varphi||_{s-1} \leq C ||f||_{s-1}
\]
for all $f \in H^{s-1}(\R^n)$ and for all $\varphi \in \mathcal G^{s-1}_\varepsilon$. By shrinking $\varepsilon$ we can assume that $\operatorname{id} + g \in \Ds^{s-1}(\R^n)$ for all $g \in H^{s-1}(\R^n;\R^n)$ with $||g||_{s-1} < \varepsilon$. Now consider for the given $u \in C\big([0,T];H^s(\R^n;\R^n)\big)$ the map
\[
 V: [0,T] \times \Ds^{s-1}(\R^n) \to \Ds^{s-1}(\R^n),\quad (t,\varphi) \mapsto u(t) \circ \varphi.
\]
From \cite{composition} we know that $V$ is a time-dependent vector field on $\Ds^{s-1}(\R^n)$, which is continuous in the time variable and $C^1$ in the $\varphi$ variable. By the existence theory for ODE's (see e.g. \cite{dieudonne}) we know that there is some $\delta > 0$ and a $\psi \in C^1\big([0,\delta],\Ds^{s-1}(\R^n)\big)$ with
\[
 \partial_t \psi = u \circ \psi \mbox{ on } [0,\delta];\quad \psi(0)=\operatorname{id}.
\]
Assume now that we have for some $0 < \delta' \leq \delta$
\[
 ||\psi(t) - \operatorname{id}||_{s-1} < \varepsilon
\]
for $0 \leq t \leq \delta'$. Note that by continuity such a $\delta'$ exists. Recall that we have
\[
 \psi(t) = \operatorname{id} + \int_0^t u(\tau) \circ \psi(\tau) \,d\tau.
\]
for all $0 \leq t \leq \delta'$. Thus we get for any $t \in [0,\delta']$
\[
 ||\psi(t) - \operatorname{id}||_{s-1} \leq C \int_0^t ||u(\tau)||_{s-1} \,d\tau \leq CM\delta'
\]
where $M=\max_{0 \leq \tau \leq T} ||u(\tau)||_s$. In particular by choosing $\delta' \leq \varepsilon/(2CM)$ we get
\begin{equation}\label{stays_in_ball}
 ||\psi(t)-\operatorname{id}||_{s-1} \leq \varepsilon/2
\end{equation}
for $0 \leq t \leq \delta'$. As $C$ is fixed, this choice of $\delta'$ just depends on $M$ and not on the particular values of $u$. Thus we see that $\forall t_0 \in [0,T]$ the ODE
\[
 \partial_t \psi = u \circ \psi;\quad \psi(t_0)=\operatorname{id}
\]
has a solution on $[t_0,t_0 + \delta'] \cap [t_0,T]$ as for these values of $t$ the condition \eqref{stays_in_ball} is preserved. Now we proceed as follows: We solve
\[
 \partial_t \psi_1 = u \circ \psi_1;\quad \psi_1(0)=\operatorname{id}
\]
on $[0,\delta']$. Then we solve
\[
 \partial_t \psi_2 = u \circ \psi_2;\quad \psi_2(\delta')=\operatorname{id}
\]
on $[\delta',2\delta']$ (without loss we can assume $2\delta' \leq T$) and define $\varphi:[0,2\delta'] \to \Ds^{s-1}(\R^n)$ by
\[
 \varphi(t) = \begin{cases} \psi_1(t), \quad & t \in [0,\delta') \\ \psi_2(t) \circ \psi_1(\delta'), \quad & t \in [\delta',2\delta'] \end{cases}.
\]
From the definition it is clear that $\varphi \in C\big([0,2\delta'];D^{s-1}(\R^n)\big)$. From the properties of $\psi_1,\psi_2$ we have
\[
 \varphi(t) = \operatorname{id} + \int_0^t u(\tau) \circ \varphi(\tau) \,d\tau
\]
for all $t \in [0,2\delta']$. Indeed on $[0,\delta']$ this is clear. For $t \in [\delta',2\delta']$ we have
\[
 \psi_2(t) = \operatorname{id} + \int_{\delta'}^t u(\tau) \circ \psi_2(\tau) \,d\tau
\]
or 
\[
 \psi_2(t)-\operatorname{id} = \int_{\delta'}^t u(\tau) \circ \psi_2(\tau) \,d\tau.
\]
Applying the continuous linear operator $R_{\psi_1(\delta')}$ to this equation we get
\[
 \psi_2(t) \circ \psi_1(\delta') = \psi_1(\delta') + \int_{\delta'}^t u(\tau) \circ \psi_2(\tau) \circ \psi_1(\delta') \,d\tau
\]
which is by definition
\[
 \varphi(t) = \varphi(\delta') + \int_{\delta'}^t u(\tau) \circ \varphi(\tau) \,d\tau
\]
showing the claim. Iterating this procedure we can construct a solution 
\[
\varphi \in C^1\big([0,T];\Ds^{s-1}(\R^n)\big).
\]
Next we want to show that we have actually 
\[
\varphi \in C^1\big([0,T];\Ds^s(\R^n)\big).
\]
Writing $\varphi=\operatorname{id} + f$ where $f \in C^1\big([0,T];H^{s-1}(\R^n)\big)$ we get by taking the differential of $\partial_t \varphi = u \circ \varphi$
\begin{equation}\label{ode_df}
 \partial_t df = du \circ \varphi \cdot (I_n + df)
\end{equation}
where $df \in C^1\big([0,T];H^{s-2}(\R^n;\R^{n \times n})\big)$ denotes the Jacobian of $f$ and $I_n$ the $n \times n$-identity matrix. As we have by the results for the composition map given in \cite{composition}
\[
 du \circ \varphi \in C^0\big([0,T];H^{s-1}(\R^n;\R^{n \times n})\big)
\]
we can view \eqref{ode_df} as a inhomogenous linear ODE with coefficients in $H^{s-1}$. By uniqueness of solutions this means that $df$ lies actually in 
\[
C^1\big([0,T];H^{s-1}(\R^n;\R^{n \times n})\big). 
\]
This show that $\varphi \in C^1\big([0,T];\Ds^s(\R^n)\big)$. Hence the claim. 
\end{proof}

\noindent
To prove Proposition \ref{prop_integration} we will need the following well-known lemmas.

\begin{Lemma}\label{lemma_interpolation}
Let $f \in H^s(\R^n)$, $s \geq 0$. Then we have the following interpolation inequality for $0 \leq s' \leq s$ and $\lambda \in (0,1)$
\begin{equation}\label{interpolation_ineq}
||f||_{\lambda s' + (1-\lambda) s} \leq ||f||_{s'}^\lambda ||f||_s^{1-\lambda}.
\end{equation}
\end{Lemma}

\begin{proof}
We have by definition
\begin{eqnarray*}
||f||^2_{\lambda s' + (1-\lambda) s} &=& \int_{\R^n} (1+|\xi|^2)^{\lambda s' + (1-\lambda) s} |\hat f(\xi)|^2 d\xi \\
&=& \int_{\R^n} (1+|\xi|^2)^{\lambda s'} |\hat f(\xi)|^{2\lambda} (1+|\xi|^2)^{(1-\lambda)s} |\hat f(\xi)|^{2(1-\lambda)} d\xi \\
\noalign{\noindent and using the H\"older inequality}\\
&\leq& ||(1+|\xi|^2)^{\lambda s'} |\hat f(\xi)|^{2\lambda}||_{L^\frac{1}{\lambda}} ||(1+|\xi|^2)^{(1-\lambda)s} |\hat f(\xi)|^{2(1-\lambda)} ||_{L^\frac{1}{1-\lambda}} \\
&=& ||f||_{s'}^{2\lambda} ||f||_s^{2(1-\lambda)} 
\end{eqnarray*}
which shows the claim.
\end{proof}

\noindent
For approximating functions by regular ones we have

\begin{Lemma}\label{lemma_approximation}
Let $f \in H^s(\R^n)$, $s \geq 0$. Let $\chi_k(D)$, $k \geq 1$, be the Fourier multiplier with symbol $\chi_k$ given by 
\[
 \chi_k(\xi)=\begin{cases} 1, \quad &|\xi| \leq k \\ 0, \quad & |\xi| > k \end{cases}
\]
Then we have $\chi_k(D) f \in H^{s+1}(\R^n)$ and
\[
 \chi_k(D) f \to f \quad \mbox{ in } H^s(\R^n)
\]
as $k \to \infty$.
\end{Lemma}

\begin{proof}
That $\chi_k(D) f \in H^{s+1}(\R^n)$ follows from
\begin{multline*}
 ||\chi_k(D)f||_{s+1}^2 = \int_{|\xi| \leq k} (1+|\xi|^2)^{s+1} |\hat f(\xi)|^2 d\xi \leq (1+k^2)^{s+1} \int_{\R^n} |\hat f(\xi)|^2 d\xi < \infty.
\end{multline*}
One has actually $\chi_k(D)f \in H^\infty(\R^n)=\cap_{s \geq 0} H^s(\R^n)$, but this is not needed here. For the second claim we write
\[
 ||\chi_k(D) f - f||_s^2 = \int_{|\xi| > k} (1+|\xi|^2)^s |\hat f(\xi)|^2 d\xi.
\]
Now by Lebesgue's dominated convergence we get
\[
 \int_{|\xi| > k} (1+|\xi|^2)^s |\hat f(\xi)|^2 d\xi \to 0
\]
as $k \to \infty$. Hence the claim.
\end{proof}

\noindent
We even have that this convergence is uniform on compact curves.

\begin{Coro}\label{coro_uniform}
Let $u \in C^0\big([0,T];H^s(\R^n)\big)$ for some $T > 0$. Then $\chi_k(D) u \in C^0\big([0,T];H^{s+1}(\R^n)\big)$ and
\[
 \sup_{0 \leq t \leq T} ||\chi_k(D) u(t) - u(t)||_s \to 0
\]
as $k \to 0$.
\end{Coro}

\begin{proof}
We will prove a slightly stronger result. We will prove that for a compact set $K \subseteq H^s(\R^n)$ we have
\[
 \chi_k(D) f \to f \quad \mbox{ in } H^s(\R^n)
\]
as $k \to \infty$ uniformly in $f \in K$. First note that $||\chi_k(D) f||_s \leq ||f||_s$. Let $\varepsilon > 0$. As $K$ is compact we have a finite set of points (let's say $M$ points) $(f_m)_{1 \leq m \leq M} \subseteq H^s(\R^n)$ such that
\[
 K \subseteq \cup_{m=1}^M B_\varepsilon(f_m)
\] 
where
\[
 B_\varepsilon(f) = \big\{ g \in H^s(\R^n) \,\big|\, ||g-f||_s < \varepsilon \big\}
\]
is the $\varepsilon$-ball in $H^s(\R^n)$ around $f$ with radius $\varepsilon$. By Lemma \ref{lemma_approximation} there is a $N$ such that
\[
 ||\chi_k(D) f_m - f_m||_s < \varepsilon 
\]
for all $k \geq N$ and $1 \leq m \leq M$. For an arbitrary $f \in K$ take a $f_j$, $1 \leq j \leq M$, with $f \in B_\varepsilon(f_j)$. With this choice we have
\begin{multline*}
 ||\chi_k(D) f - f||_s \leq ||\chi_k(D)f - \chi_k(D) f_j||_s + ||\chi_k(D) f_j - f_j ||_s + ||f_j - f||_s < 3 \varepsilon
\end{multline*}
for all $k \geq N$. This proves the claim for the compact set $K$. Now as the image of the curve $u$ is compact we get the desired result.
\end{proof}

\noindent
We know that there is some $\varepsilon > 0$ such that for all $g \in H^s(\R^n;\R^n)$ with $||g||_s < \varepsilon$ we have $\operatorname{id}+g \in \Ds^s(\R^n)$. Denote this set by $\mathcal G_\varepsilon^s$, i.e.
\[
 \mathcal G_\varepsilon^s = \big\{ \varphi \in \Ds^s(\R^n) \, \big| \, ||\varphi - \operatorname{id}||_s < \varepsilon \big\}.
\] 
By Lemma \ref{lemma_linear_growth} we get (by shrinking $\varepsilon$ if necessary) for all $\varphi \in \mathcal G_\varepsilon^s$ 
\begin{equation}\label{uniform_estimate1}
 ||f \circ \varphi||_{s-1} \leq C ||f||_{s-1}, \quad \forall f \in H^{s-1}(\R^n;\R^n)
\end{equation}
and
\begin{equation}\label{uniform_estimate2}
 ||f \circ \varphi||_s \leq C ||f||_s, \quad \forall f \in H^s(\R^n;\R^n)
\end{equation}
for some $C>0$. We further assume by making $0 < \varepsilon < 1$ small enough that we have $\det(d_x \varphi) > \varepsilon$ for all $x \in \R^n$ and for all $\varphi \in \mathcal G_\varepsilon^s$. Because of the Sobolev imbedding this is possible. Now with this choice of $\varepsilon$ resp. $\mathcal G_\varepsilon^s$ we prove the following Lipschitz type estimate.

\begin{Lemma}\label{lemma_lipschitz_type}
There is $\tilde C > 0$ such that for any $\varphi,\psi \in \mathcal G_\varepsilon^s$
\[
 ||f \circ \varphi - f \circ \psi||_{s-1} \leq \tilde C ||f||_s ||\varphi - \psi||_{s-1},\quad \forall f \in H^s(\R^n).
\]
\end{Lemma}

\begin{proof}
By the fundamental lemma of calculus we have pointwise
\begin{eqnarray}
\nonumber
 f \circ \varphi - f \circ \psi &=& \int_0^1 \partial_t \left( f\big(\psi + t(\varphi - \psi)\big) \right) dt \\
\label{fundamental_lemma}
&=& \int_0^1 \nabla f\big(\psi + t (\varphi - \psi)\big) (\varphi - \psi) dt
\end{eqnarray}
As $t \mapsto \psi +t (\varphi - \psi)$ is a continuous curve in $\mathcal G_\varepsilon^s$ we see that the integrand is a continuous curve $H^{s-1}(\R^n;\R^n)$. Indeed $\varphi - \psi \in H^s(\R^n;\R^n)$ and $H^{s-1}$ is a Banach algebra. Thus we see that \eqref{fundamental_lemma} is an identity in $H^{s-1}(\R^n;\R^n)$. Therefore we have for some $\tilde C >0$
\begin{eqnarray*}
||f \circ \varphi - f \circ \psi||_{s-1} &\leq& \int_0^1 \tilde C ||\nabla f \big(\psi + t (\varphi - \psi)\big)||_{s-1} ||\varphi - \psi||_{s-1} dt \\
&\leq& \tilde C ||f||_s ||\varphi - \psi||_{s-1}
\end{eqnarray*}
where we used \eqref{uniform_estimate1} implying
\[
 ||\nabla f \big(\psi + t (\varphi - \psi)\big)||_{s-1} \leq C ||\nabla f||_{s-1} \leq C ||f||_s
\] 
and the Banach algebra property of $H^{s-1}(\R^n)$. This finishes the proof. 
\end{proof}

\noindent
Now we can prove the main proposition. We will do this using some ''energy'' estimates. We take $\mathcal G_\varepsilon^s$ as in Lemma \ref{lemma_lipschitz_type} 

\begin{proof}[Proof of Proposition \ref{prop_integration}]
Let $u \in C^0\big([0,T];H^s(\R^n;\R^n)$ be the given continuous vector field. We define $u_k=\chi_k(D)u$. We know by Corollary \ref{coro_uniform} that $u_k(t) \to u(t)$ in $H^s$ uniformly in $t \in [0,T]$. By Lemma we know that $u_k \in C^0\big([0,T];H^{s+1}(\R^n;\R^n)\big)$. Now Lemma \ref{lemma_integration} gives us corresponding flows $\varphi_k \in C^1\big([0,T];\Ds^{s+1}(\R^n)\big)$ solving
\[
 \partial_t \varphi_k = u_k \circ \varphi_k \mbox{ on } [0,T];\quad \varphi_k(0)=\operatorname{id}.
\]
We will show first that $\varphi_k$ converges at least on some short time interval $[0,\delta]$ to the desired solution. Consider the integral relation
\[
 \varphi_k(t) = \operatorname{id} + \int_0^t u_k(\tau) \circ \varphi_k(\tau) \,d\tau, \quad t \in [0,T].
\]
We reason as in the proof of Lemma \ref{lemma_integration}. For $k \geq 1$ fixed, assume that $\varphi_k(t) \in \mathcal G_\varepsilon^s$ for all $0 \leq t \leq \delta'$, for some $\delta' > 0$. Then we have for $t \in [0,\delta']$
\[
 ||\varphi_k(t)-\operatorname{id}||_s \leq \int_0^t ||u_k(\tau) \circ \varphi_k(\tau)||_s \,d\tau \leq C \int_0^t ||u_k(\tau)||_s \, d\tau
\]
where we used \eqref{uniform_estimate2}. Now as we have $u_k \to u$ uniformly in $t \in [0,T]$ there is some $M >0$ with
\[ 
 ||u_k(t)||_s < M
\]
for all $t \in [0,T]$ and for all $k \geq 1$. Thus we see that for $\delta \leq \frac{\varepsilon}{2CM}$ we have $||\varphi_k(t) - \operatorname{id}||_s < \varepsilon$ for all $t \in [0,\delta]$, i.e. we have $\varphi_k(t) \in \mathcal G_\varepsilon^s$. Now we will show that $\varphi_k$ converges on $[0,\delta]$. We have for $t \in [0,\delta]$
\begin{eqnarray*}
 \varphi_k(t)-\varphi_j(t) &=& \int_0^t u_k \circ \varphi_k - u_j \circ \varphi_j \,d\tau \\
&=& \int_0^t u_k \circ \varphi_k - u_j \circ \varphi_k \,d\tau + \int_0^t u_j \circ \varphi_k - u_j \circ \varphi_j \,d\tau.
\end{eqnarray*}
Taking the $H^{s-1}$-norm we get for any $t \in [0,\delta]$
\[
||\varphi_k(t)- \varphi_j(t)||_{s-1} \leq C \int_0^t ||u_k - u_j||_s \,d\tau + C \int_0^t ||u_j||_s ||\varphi_k- \varphi_j||_{s-1} \,d\tau 
\]
where we used Lemma \ref{lemma_lipschitz_type}. Thus from Gronwall's lemma we get
\[
 ||\varphi_k(t)-\varphi_j(t)||_{s-1} \leq \left[C \int_0^\delta ||u_k-u_j||_s\right] e^{\delta C M}
\]
for all $t \in [0,\delta]$. Thus we see that $\varphi_k - \varphi_j$ is Cauchy in $C^0\big([0,\delta];H^{s-1}(\R^n;\R^n)\big)$. Now take a $\lambda \in (0,1)$ with
\[
 s'=\lambda (s-1) + (1-\lambda) s > n/2+1.
\]
As we have $s > n/2+1$ such a $\lambda$ exists. We then have by Lemma \ref{lemma_interpolation}
\begin{eqnarray*}
 ||\varphi_k(t)-\varphi_j(t)||_{s'} &\leq& ||\varphi_k(t)-\varphi_j(t)||_{s-1}^\lambda ||\varphi_k(t)-\varphi_j(t)||_s^{1-\lambda} \\
&\leq& ||\varphi_k(t)-\varphi_j(t)||_{s-1}^\lambda \left(||\varphi_k(t)-\operatorname{id}||_s + ||\varphi_j(t)-\operatorname{id}||_s\right)^{1-\lambda} \\
&\leq& ||\varphi_k(t)-\varphi_j(t)||_{s-1}^\lambda (2\varepsilon)^{1-\lambda}
\end{eqnarray*}  
showing that $\varphi_k$ converges in $H^{s'}$ on $[0,\delta]$. Thus there exists a $\varphi$ with $\varphi- \operatorname{id} \in C^0\big([0,\delta];H^{s'}(\R^n;\R^n)\big)$ such that we have 
\[
 ||\varphi_k(t) - \varphi(t)||_{s'} \to 0
\]
uniformly in $t \in [0,\delta]$. As $s' > n/2 +1$ we have by the Sobolev imbedding for all $x \in \R^n$
\[
 \det(d_x \varphi) = \lim_{k \to \infty} \det(d_x \varphi_k) \geq \varepsilon > 0.
\]
Hence $\varphi \in C^0\big([0,\delta];\Ds^{s'}(\R^n)\big)$. We claim that $\varphi=\tilde \varphi$ on $[0,\delta]$. Recall that we denote by $\tilde \varphi$ the flow of the vector field $u:[0,T]\times \R^n \to \R^n$. To show that $\varphi$ and $\tilde \varphi$ agree on $[0,\delta]$ consider for $x \in \R^n$ and $t \in [0,\delta]$
\begin{equation}\label{integral_limit}
\varphi_k(t,x) = x + \int_0^t u_k\big(\tau,\varphi_k(\tau,x)\big) \,d\tau.
\end{equation}
By the Sobolev imbedding we have (denoting by $|\cdot|$ the euclidean norm in $\R^n$)
\begin{multline*}
\left| u_k\big(t,\varphi_k(t,x)\big) - u\big(t,\varphi(t,x)\big)\right| \leq \left|u_k\big(t,\varphi_k(t,x)\big) - u\big(t,\varphi_k(t,x)\big)\right| \\
+ \left|u\big(t,\varphi_k(t,x)\big) - u\big(t,\varphi(t,x)\big)\right| \leq C ||u_k(t) - u(t)||_{s'} + C ||u||_{s'} ||\varphi_k(t) - \varphi(t)||_{s'}
\end{multline*}
which goes to $0$ uniformly in $t \in [0,\delta]$. Thus taking the limit in \eqref{integral_limit} we arrive at
\begin{equation}\label{limit_integral}
 \varphi(t,x) = x + \int_0^t u\big(\tau,\varphi(\tau,x)\big) \,d\tau.
\end{equation}
By continuity of the composition in $H^{s'}$ we see that the identity \eqref{limit_integral} holds in $H^{s'}$, i.e. we have
\begin{equation}\label{hs_identity}
\varphi(t) = \operatorname{id} + \int_0^t u(\tau) \circ \varphi(\tau) \,d\tau.
\end{equation}
Taking the differential in \eqref{limit_integral} and denoting by $I_n$ the $n \times n$ identity matrix, we have
\[
 d\varphi(t,x) = I_n + \int_0^t du\big(\tau,\varphi(\tau,x)\big) d\varphi(\tau,x) \,d\tau.
\]
Thus $dg:=d\varphi - I_n$ solves for fixed $x \in \R^n$ the ODE
\begin{equation}\label{ode_dg}
\partial_t dg = du \circ \varphi + du \circ \varphi \cdot dg.
\end{equation}
From \cite{composition}, as $s' > n/2+1$ and $s' \geq s-1$, we know that
\[
du \circ \varphi \in C^0\big([0,\delta];H^{s-1}(\R^n;\R^{n \times n})\big). 
\]
Since $H^{s-1}$ is an algebra, we can view \eqref{ode_dg} as a linear inhomogeneous ODE with coefficients in $H^{s-1}$. Thus $dg$ lies actually in 
\[
 C^1\big([0,\delta];H^{s-1}(\R^n;\R^{n \times n})\big).
\]
Thus we get $\varphi \in C^1\big([0,\delta];\Ds^s(\R^n)\big)$ and the identity \eqref{hs_identity} holds in $\Ds^s(\R^n)$. To get $\varphi$ on the whole time interval $[0,T]$ we proceed as in the proof of Lemma \ref{lemma_integration}. As $\delta$ just depends on $M$ we can extend $\varphi$ by $\delta$-steps. After finitely many steps we end up with the desired flow $\varphi \in C^1\big([0,T];\Ds^s(\R^n)\big)$ solving
\[
\partial_t \varphi = u \circ \varphi \mbox{ on } [0,T];\quad \varphi(0)=\operatorname{id}
\]
and this proves the proposition.
\end{proof}

\section{Alternative formulation}\label{section_alternative}

In this section we show that \eqref{alternative_E} gives an alternative formulation of \eqref{E}. First we prove

\begin{Lemma}\label{lemma_dont_miss}
Let $(u,p)$ be a solution to \eqref{E}. Then $u$ is a solution to \eqref{alternative_E}.
\end{Lemma}

\begin{proof}[Proof of Lemma \ref{lemma_dont_miss}]
Taking the divergence in \eqref{E} one has
\[
 -\Delta p = -\operatorname{div} (\nabla p)= \sum_{i,k=1}^n \partial_i u_k \partial_k u_i.
\]
\noindent
On the other hand 
\begin{eqnarray*}
\label{div_B_splitted}
\operatorname{div} \nabla B(u)&=&\Delta B(u) = \sum_{i,k=1}^n \chi(D) \big(\partial_i \partial_k (u_i u_k)\big) + \big(1-\chi(D)\big) (\partial_i u_k \partial_k u_i) \\
\noalign{\noindent and hence using that $\operatorname{div} u =0$}
\label{div_B}
 \operatorname{div} \nabla B(u) &=& \sum_{i,k=1}^n \partial_i u_k \partial_k u_i. 
\end{eqnarray*}
Thus $-\Delta p=\Delta B(u)$. Therefore each component of $\nabla B(u)+\nabla p$ is harmonic. As $\nabla B(u), \nabla p$ vanish at infinity we have actually $\nabla B(u)=-\nabla p$. Thus $u$ solves \eqref{alternative_E}.
\end{proof}

Now let us prove the converse of Lemma \ref{lemma_dont_miss}.

\begin{Lemma}\label{lemma_closedness}
Let $u_0 \in H_\sigma^s(\R^n;\R^n)$ and assume that $u \in C^0\big([0,T];H^s(\R^n;\R^n)\big)$ is a solution to \eqref{alternative_E} for some $T>0$ with initial value $u_0$. Then
\[
 u(t) \in H_\sigma^s(\R^n;\R^n), \quad \forall t \in [0,T]
\]
and $\big(u,-B(u)\big)$ satisfies (S1)-(S3).
\end{Lemma}

\begin{proof}[Proof of Lemma \ref{lemma_closedness}]
To show that $\operatorname{div} u(t)=0$ for any $0 \leq t \leq T$ it suffices to prove that $\partial_t \operatorname{div} u = 0$. Applying $\operatorname{div}$ to \eqref{alternative_E} and using the assumption $\operatorname{div} u_0=0$ one gets
\begin{equation}\label{div_int}
\operatorname{div}u = \int_0^t \sum_{i,k=1}^n \operatorname{div}\nabla B(u)  - \operatorname{div} \big((u \cdot \nabla)u\big)\, d\tau.
\end{equation}
We have
\[
 \operatorname{div} \nabla B(u) = \Delta B(u) = \sum_{i,k=1}^n \chi(D) \big(\partial_i \partial_k (u_i u_k)\big) + \big(1-\chi(D)\big) (\partial_i u_k \partial_k u_i).
\]
Note that
\[
\sum_{i,k=1}^n \partial_i \partial_k (u_i u_k) = 2 (u \cdot \nabla) \operatorname{div} u + (\operatorname{div} u)^2 + \sum_{i,k=1}^n \partial_i u_k \partial_k u_i.
\]
Therefore
\begin{multline}
\chi(D) \sum_{i,k=1}^n \partial_i \partial_k (u_i u_k) + \big(1-\chi(D)\big) \sum_{i,k=1}^n \partial_i u_k \partial_k u_i \\
= \chi(D) 2 (u \cdot \nabla) \operatorname{div} u + \chi(D) (\operatorname{div} u)^2 + \sum_{i,k=1}^n \partial_i u_k \partial_k u_i.
\end{multline}
Furthermore
\[
 -\operatorname{div}\big( (u \cdot \nabla) u\big) = -(u \cdot \nabla)\operatorname{div} u - \sum_{i,k=1}^n \partial_i u_k \partial_k u_i.
\]
Substituting the two identities above into \eqref{div_int} one gets in $L^2$ (as $n \geq 2$, one has $s > n/2+1 \geq 2$)
\begin{equation}\label{u_derivative}
 \partial_t \operatorname{div} u = \chi(D)\big( 2 (u \cdot \nabla)\operatorname{div} u + (\operatorname{div} u)^2\big) - (u \cdot \nabla) \operatorname{div}u.
\end{equation}
Denoting by $\langle \cdot,\cdot \rangle_{L^2}$ the inner product $\int_{\R^n} f g\,dx$ for two real valued $L^2$-functions we have
\begin{equation}\label{L2_derivative}
 \frac{1}{2} \partial_t \langle \operatorname{div}u,\operatorname{div} u \rangle_{L^2} = 2I + II - III
\end{equation}
where
\begin{eqnarray*}
 I &=& \langle \operatorname{div} u,\chi(D) \big((u \cdot \nabla)\operatorname{div}u\big)\rangle_{L^2}\\ 
II &=& \langle \operatorname{div} u,\chi(D) (\operatorname{div} u)^2\rangle_{L^2}\\
III &=& \langle \operatorname{div} u, (u \cdot \nabla) \operatorname{div} u\rangle_{L^2}
\end{eqnarray*}
We now estimate the terms on the right-hand side of \eqref{L2_derivative} seperately for each fixed $0 \leq t \leq T$. To estimate the term $III$ we integrate by parts
\begin{eqnarray*}
 \langle \operatorname{div}u, (u \cdot \nabla) \operatorname{div} u \rangle_{L^2} &=& -\langle \sum_{k=1}^n \partial_k (u_k \operatorname{div} u), \operatorname{div}u \rangle_{L^2}\\
 &=& -\langle (\operatorname{div}u)^2,\operatorname{div}u \rangle_{L^2} - \langle \operatorname{div}u, (u \cdot \nabla) \operatorname{div}u \rangle_{L^2}.
\end{eqnarray*}
Thus we get
\begin{equation}\label{integ_by_parts}
 \langle \operatorname{div}u, (u \cdot \nabla)\operatorname{div}u \rangle_{L^2} = -\frac{1}{2} \langle (\operatorname{div}u)^2, \operatorname{div}u \rangle_{L^2}.
\end{equation}
Using the imbedding $H^{s-1}(\R^n) \hookrightarrow C_0^0(\R^n)$ it follows from \eqref{integ_by_parts} that 
\[
 |(\operatorname{div} u, (u \cdot \nabla) \operatorname{div} u)_{L^2}| \leq \frac{1}{2} ||\operatorname{div}u||_{L^\infty} ||\operatorname{div} u||^2_{L^2}.
\]
To estimate the term $I$ use the $L^2$-symmetry of $\chi(D)$ to get 
\begin{multline*}
 \langle \operatorname{div}u, \chi(D) (u \cdot \nabla)\operatorname{div} u \rangle_{L^2} = \langle \chi(D) \operatorname{div}u , (u \cdot \nabla)\operatorname{div}u \rangle_{L^2}\\
 = -\langle \sum_{k=1}^n \partial_k (u_k \chi(D) \operatorname{div} u), \operatorname{div} u \rangle_{L^2}\\
 = -\langle (\operatorname{div}u) \chi(D) \operatorname{div} u, \operatorname{div} u \rangle_{L^2} 
 - \langle (u \cdot \nabla) \chi(D) \operatorname{div}u, \operatorname{div}u \rangle_{L^2}.
\end{multline*}
Hence 
\[
 2I + II = -\langle \chi(D)\operatorname{div} u,(\operatorname{div}u)^2\rangle_{L^2} - 2 \langle (u \cdot \nabla) \chi(D) \operatorname{div} u,\operatorname{div} u\rangle_{L^2}
\]
or
\begin{eqnarray*}
|2I+II| &\leq & ||\chi(D) \operatorname{div} u||_{L^2} ||\operatorname{div}u||_{L^\infty} ||\operatorname{div} u||_{L^2} + 2 ||(u \cdot \nabla) \chi(D) \operatorname{div} u||_{L^2} ||\operatorname{div} u||_{L^2}\\
&\leq & ||\operatorname{div} u||_{L^\infty} ||\operatorname{div} u||_{L^2}^2 + 2 ||(u \cdot \nabla) \chi(D) \operatorname{div} u||_{L^2} ||\operatorname{div} u||_{L^2}.
\end{eqnarray*}
Using the smoothing property \ref{chi_smoothing} and the imbedding $H^s(\R^n) \hookrightarrow C_0^0(\R^n)$ once more leads to
\[
 ||(u \cdot \nabla )\chi(D) \operatorname{div} u||_{L^2} \leq \sqrt{2} ||u||_{L^\infty} ||\operatorname{div}u||_{L^2}.
\]
Summarizing the above inequalities, we conclude that for any $0 \leq t \leq T$
\[
 \partial_t ||\operatorname{div}u||_{L^2}^2 \leq 3\big(||\operatorname{div}u||_{L^\infty} + ||u||_{L^\infty}\big) ||\operatorname{div}u||_{L^2}^2.
\]
Using the imbedding $H^s(\R^n) \hookrightarrow C_0^1(\R^n)$ we conclude that
\[
 C:= \sup_{0 \leq t \leq T} \big(||\operatorname{div} u(t)||_{L^\infty} + ||u(t)||_{L^\infty}\big) < \infty.
\]
As $\operatorname{div}u(0)=0$ we then get by Gronwall's inequality (see e.g. \cite{majda}) that $\operatorname{div}u(t)=0$ for all $t \in [0,T]$.
\end{proof}

We combine the results proved in Lemma \ref{lemma_dont_miss} and Lemma \ref{lemma_closedness} in the following

\begin{Prop}\label{prop_alternative}
For any solution $u \in C^0\big([0,T];H^s(\R^n;\R^n)\big)$ with $u(0)=u_0 \in H^s_\sigma(\R^n;\R^n)$ of \eqref{alternative_E} the pair $\big(u,-B(u))$ is a solution of \eqref{E}. Conversely any solution to \eqref{E} gives rise to a solution of \eqref{alternative_E}.
\end{Prop}

\section{The submanifold $\Ds^s_\mu(\R^n)$}\label{section_submanifold}

Throughout this section we assume as usual $s > n/2+1$ with $n \geq 2$. We will prove Theorem \ref{th_submanifold} saying that $\Ds^s_\mu(\R^n)$ is a closed analytic submanifold of $\Ds^s(\R^n)$. The most natural way to prove this statement is to consider the analytic map
\begin{equation}\label{map_F}
 \varphi \mapsto \left[F(\varphi):\R^n \to \R,\quad x \mapsto \det(d_x \varphi) - 1\right]
\end{equation}
for $\varphi$ in $\Ds^s(\R^n)$ -- see the proof of the corresponding result for $\Ds^s(M)$, $M$ a compact manifold, of Ebin and Marsden \cite{ebin_marsden}. We clearly have $\Ds^s_\mu(\R^n)=F^{-1}(0)$. Using the Banach algebra property of $H^{s-1}(\R^n)$ one shows that $F$ takes values in $H^{s-1}(\R^n)$ and is analytic. In particular $\Ds_\mu^s(\R^n)$ is a closed subset of $\Ds^s(\R^n)$ and it remains to show that $0 \in H^{s-1}(\R^n)$ is a regular value of $F$. The differential of $F$ at $\operatorname{id} \in \Ds^s_\mu(\R^n)$ is given by 
\[
 d_{\operatorname{id}}F:H^s(\R^n;\R^n) \to H^{s-1}(\R^n),\quad f \mapsto \operatorname{div} f
\]
which is however not surjective.

\begin{Lemma}\label{lemma_not_surjective}
The map
\[
 \operatorname{div}:H^s(\R^n;\R^n) \to H^{s-1}(\R^n),\quad f \mapsto \operatorname{div} f
\]
is not surjective.
\end{Lemma}

\begin{proof}
Assume that $\operatorname{div}$ is surjective. As $H^s_\sigma(\R^n;\R^n)$ is by definition the null space of $\operatorname{div}$, the map
\[
 \Psi:H^s_\sigma(\R^n;\R^n)^\perp \to H^{s-1}(\R^n),\quad f \mapsto \operatorname{div} f
\]
is then a bijection. Here $H^s_\sigma(\R^n;\R^n)^\perp$ is the orthogonal complement of $H^s_\sigma(\R^n;\R^n)$ in $H^s(\R^n;\R^n)$ with respect to the inner product $\langle \cdot,\cdot \rangle_s$. By the open mapping theorem $\Psi$ has a continuous inverse denoted by $\Phi$
\[
 \Phi:H^{s-1}(\R^n) \to H^s_\sigma(\R^n;\R^n)^\perp.
\]
In particular it means that there is a constant $C > 0$ so that
\begin{equation}\label{continuous_inverse}
||\Phi(w)||_s \leq C ||w||_{s-1}, \qquad \forall w \in H^{s-1}(\R^n).
\end{equation}
We then get for any $w \in H^{s-1}(\R^n)$ by integration by parts
\[
 \langle  w, w \rangle_{s-2} = \langle \Psi \Phi(w), w \rangle_{s-2} = -\sum_{j=1}^n \langle \Phi_j(w),\partial_j w \rangle_{s-2}.
\]
Applying Cauchy-Schwarz we get for any $w \in H^{s-1}(\R^n)$
\begin{eqnarray}
\nonumber
 ||w||_{s-2}^2 &\leq& ||\Phi(w)||_{s-2} ||\nabla w||_{s-2} \\
\label{wrong_inequality}
&\leq& C ||w||_{s-1} ||\nabla w||_{s-2}
\end{eqnarray}
where we used \eqref{continuous_inverse}. We claim that the inequality \eqref{wrong_inequality} cannot hold. To see it take an element $w \in H^{s-1}(\R^n)$ with $||w||_{L^2}=1$ whose Fourier transform $\hat w$ is supported in the unit ball, $\operatorname{supp} \hat w \subseteq B_1(0)$. Define for $k \geq 1$ and $x \in \R^n$
\[
w_k(x) = w\left(\frac{x}{k}\right).
\]
Using that $\widehat{w_k}(\xi)=k^n \hat w(k \xi)$, one has
\begin{eqnarray*}
 ||w_k||_{s-2}^2 &=& \int_{\R^n} (1+|\xi|^2)^{s-2} k^{2n} |\hat w(k\xi)|^2 \,d\xi \\
\noalign{\noindent and by the change of variable $\eta:=k\xi$}
 &=& \int_{|\eta| \leq 1} (1+|\frac{\eta}{k}|^2)^{s-2} k^n |\hat w(\eta)|^2 \,d\eta \geq k^n ||w||^2_{L^2}=k^n.
\end{eqnarray*}
Analogously we have
\[
 ||w_k||_{s-1}^2 = \int_{|\eta| \leq 1} (1+|\frac{\eta}{k}|^2)^{s-1} k^n |\hat w(\eta)|^2 \,d\eta \leq 2^{s-1} k^n ||w||_{L^2}^2=2^{s-1}k^n.
\]
Similarly we have for $1 \leq j \leq n$
\begin{eqnarray*}
 ||\partial_j w_k||_{s-2}^2 &=& \int_{\R^n} (1+|\xi|^2)^{s-2} \xi_j^2 k^{2n} |\hat w(k \xi)|^2 \,d\xi \\
&=& \int_{|\eta|\leq 1} (1+|\frac{\eta}{k}|^2)^{s-2} \frac{\eta_j^2}{k^2} k^n |\hat w(\eta)|^2 \,d\eta\\
&\leq& 2^{s-2} k^{n-2} ||w||_{L^2}^2=2^{s-2}k^{n-2}.
\end{eqnarray*}
So 
\[
 ||w_k||_{s-1} ||\nabla w||_{s-2} \leq (2^{s-1}  k^n)^{1/2} \cdot (2^{s-2}k^{n-2})^{1/2} =2^{s-3/2} k^{n-1} \mbox{ and } ||w||_{s-2}^2 \geq k^n.
\]
Thus for $k$ large the inequality \eqref{wrong_inequality} cannot hold. This shows that the assumption that $\operatorname{div}$ is surjetive is wrong. 
\end{proof}

Lemma \ref{lemma_not_surjective} shows that $0$ is not a regular value of $F$. To prove Theorem \ref{th_submanifold} we therefore have to argue differently then Ebin and Marsden in \cite{ebin_marsden}. The key idea is to use the exponential map as a parametrization of $\Ds^s_\mu(\R^n)$. We want to show that near $\operatorname{id} \in \Ds^s_\mu(\R^n)$ there exists a neighborhood $V$ of $0$ in $H^s_\sigma(\R^n;\R^n)$ which $\exp$ maps bijectively onto a neighborhood of $\operatorname{id}$ in $\Ds^s_\mu(\R^n)$. In the following we denote by $U^s_{\exp} \subseteq H^s(\R^n;\R^n)$ the domain of the exponential map. 

\begin{Prop}\label{prop_local_manifold}
There is a neighborhood $\tilde U \subseteq U^s_{\exp}$ of $0$ such that
\[
 \exp\big(\tilde U \cap H^s_\sigma(\R^n;\R^n)\big) = \exp(\tilde U) \cap \Ds^s_\mu(\R^n).
\]
Moreover $\left. \exp \right|_{\tilde U}$ is an analytic diffeomorphism onto its image.
\end{Prop}

First we have to make some preparations for the proof of Proposition \ref{prop_local_manifold}. In the following we denote as usual by $u(t;u_0)$ the solution of \eqref{RE} at time $t$ with initial value $u_0$. By the proof of Theorem \ref{th_main}, $\; u(t;u_0)$ is well-defined on $[0,1] \times U^s_{\exp}$. We state without proof the following quite well-known lemma (see e.g. \cite{majda})

\begin{Lemma}\label{det_derivative}
Let $u:[0,1] \times \R^n \to \R^n$ be a $C^1$-vector field admitting a flow $\varphi:[0,1] \times \R^n \to \R^n$, i.e. a $C^1$-map satisfying
\[
 \partial_t \varphi(t,x)=u(t,\varphi(t,x)) \quad \mbox{and} \quad \varphi(0,x)=x
\]
for any $(t,x) \in [0,1] \times \R^n$. Then for all $(t,x) \in [0,1] \times \R^n$
\[
 \partial_t \det\big(d_x \varphi(t,x)\big) = (\operatorname{div} u)\big(t,\varphi(t,x)\big) \cdot \det\big(d_x\varphi(t,x)\big)
\]
or, in integrated form,
\begin{equation}\label{det_formula}
\det\big(d_x\varphi(t,x)\big) = e^{\int_0^t (\operatorname{div} u)(\tau,\varphi(\tau,x)) \,d\tau}.
\end{equation}
\end{Lemma}

As a consequence of Lemma \ref{det_derivative} and the proof of Theorem \ref{th_main} we get

\begin{Coro}\label{exp_restricted}
The exponential map $\exp$ maps the divergence free vector fields into the volume-preserving diffeomorphisms, i.e. 
\[
 \exp\big(U^s_{\exp} \cap H^s_\sigma(\R^n;\R^n)\big) \subseteq \Ds^s_\mu(\R^n).
\]
\end{Coro}

\begin{Lemma}\label{lemma_small_u}
For any $\varepsilon > 0$ there is a neighborhood $\tilde U \subseteq U^s_{\exp}$ of $0$ such that
\[
 ||u(t;u_0)||_s < \varepsilon
\]
for all $t \in [0,1]$ and for all $u_0 \in \tilde U$.
\end{Lemma}

\begin{proof}[Proof of Lemma \ref{lemma_small_u}]
Since by the proof of Theorem \ref{th_main}
\[
 [0,1] \times U^s_{\exp} \to H^s(\R^n;\R^n), \quad (t,u_0) \mapsto u(t;u_0)
\]
is continuous and $u(t;0)=0$ for all $t \in [0,1]$, the claim follows by the compactness of $[0,1]$.
\end{proof}

\begin{Lemma}\label{lemma_linear_u}
There is a neighborhood $\tilde U \subseteq U^s_{\exp}$ of $0$ and a constant $C>0$ such that we have for any $0 \leq t \leq 1$ and any $u_0 \in \tilde U$
\[
 ||\operatorname{div} u(t;u_0)||_{s-1} \leq C ||\operatorname{div} u_0||_{s-1}.
\]
\end{Lemma}

\begin{proof}[Proof of Lemma \ref{lemma_linear_u}]
Choose $\tilde U$ to be a small ball around $0$ contained in $U^s_{\exp}$ so that on the one hand by Lemma \ref{lemma_small_u}
\begin{equation}\label{assumption1}
 ||u(t;u_0)||_s \leq 1, \quad \forall 0 \leq t \leq 1,\; \forall u_0 \in \tilde U
\end{equation}
and on the other hand, by Lemma \ref{lemma_linear_growth}, for some $C_1 >0$, for any $\varphi \in \exp(\tilde U)$
\begin{equation}\label{assumption2}
 ||R_\varphi f||_{s-1} \leq C_1 ||f||_{s-1} \mbox{ and } ||R_\varphi^{-1} f||_{s-1} \leq C_1 ||f||_{s-1},\; \forall f \in H^{s-1}(\R^n).
\end{equation}
Denote by $\varphi(\cdot;u_0)$ the flow corresponding to $u(\cdot;u_0)$. By the chain rule for $s$ sufficiently large one has
\begin{equation}\label{t_derivative_div}
 \partial_t \big((\operatorname{div} u)\circ \varphi\big) = R_\varphi\left( \partial_t \operatorname{div}u + (u \cdot \nabla)\operatorname{div}u\right).
\end{equation}
Approximate $u(\cdot;u_0)$ by $(u_k)_{k \geq 1} \subseteq C^1\big([0,1];H^{s+1}(\R^n;\R^n))$ in the norm of the space
\[
C^0\big([0,1];H^s(\R^n;\R^n)\big) \cap C^1\big([0,1];H^{s-1}(\R^n;\R^n)\big).
\]
Then for any $k \geq 1$, one has in $C^0\big([0,1];H^{s-1}(\R^n;\R^n)\big)$
\[
 \partial_t \big((\operatorname{div} u_k)\circ \varphi\big) = R_\varphi\left( \partial_t \operatorname{div}u_k + (u \cdot \nabla)\operatorname{div}u_k\right)
\]
In particular the identity holds in $C^0\big([0,1];H^{s-2}(\R^n;\R^n)\big)$. Letting $k \to \infty$ on both sides of the latter identity leads to
\[
 \partial_t \big((\operatorname{div} u)\circ \varphi\big) = R_\varphi\left( \partial_t \operatorname{div}u + (u \cdot \nabla)\operatorname{div}u\right).
\]
Substituting formula \eqref{u_derivative} for $\partial_t \operatorname{div} u$ one gets
\begin{equation}\label{t_derivative_explicit}
 \partial_t \big((\operatorname{div} u) \circ \varphi\big) = R_\varphi \left( \chi(D) \big( 2(u \cdot \nabla) \operatorname{div} u + (\operatorname{div} u)^2 \big) \right).
\end{equation}
Integrating \eqref{t_derivative_explicit} with respect to $t$ yields
\[
 \operatorname{div} u(t) = R_{\varphi(t)}^{-1} \left(\operatorname{div} u_0 + \int_0^t R_{\varphi(\tau)} \left(\chi(D)\Big(2(u(\tau) \cdot \nabla) \operatorname{div} u(\tau) + (\operatorname{div} u(\tau))^2\Big)\right) \,d\tau \right).
\]
Using \eqref{assumption2} we get
\begin{multline}
\label{divu_triangle}
 ||\operatorname{div} u(t)||_{s-1} \leq C_1 ||\operatorname{div} u_0||_{s-1} \\
 + C_1^2 \int_0^t ||\chi(D)\big(2(u(\tau) \cdot \nabla) \operatorname{div} u(\tau)\big)||_{s-1}  + ||\chi(D)(\operatorname{div} u(\tau))^2||_{s-1} \,d\tau
\end{multline}
For the first expression under the integral sign we have by \ref{chi_smoothing} for all $0 \leq \tau \leq 1$
\[
 ||\chi(D)\big(2(u(\tau) \cdot \nabla) \operatorname{div} u(\tau)\big)||_{s-1} \leq 2\sqrt{2} ||\big(u(\tau) \cdot \nabla\big) \operatorname{div} u(\tau)||_{s-2}.
\]
Multiplication properties of Sobolev functions imply that there exists a constant $C_2 > 0$ such that for all $0 \leq \tau \leq 1$
\[
 \||\big(u(\tau) \cdot \nabla\big) \operatorname{div} u(\tau)||_{s-2} \leq C_2 ||u(\tau)||_s ||\operatorname{div} u(\tau)||_{s-1}.
\]
Combined with \eqref{assumption1} we thus have proved that
\[
  ||\chi(D)\big(2(u(\tau) \cdot \nabla) \operatorname{div} u(\tau)\big)||_{s-1} \leq 2 \sqrt{2} C_2 ||\operatorname{div} u(\tau)||_{s-1}.
\]
For the second expression in the integrand in \eqref{divu_triangle} the Banach algebra property of $H^{s-1}(\R^n)$ says that there exists an absolute constant $C_3 > 0$ so that
\[
 ||\chi(D) \big(\operatorname{div} u(\tau)\big)^2||_{s-1} \leq ||\big(\operatorname{div}u(\tau)\big)^2||_{s-1}\leq C_3 ||\operatorname{div} u(\tau)||_{s-1}^2,\; \forall 0 \leq \tau \leq 1,
\]
or using that $||\operatorname{div} u(\tau)||_{s-1} \leq ||u(\tau)||_s \leq 1$ one concludes that
\[
 ||\chi(D) \big(\operatorname{div} u(\tau)\big)^2||_{s-1} \leq C_3 ||\operatorname{div} u(\tau)||_{s-1},\; \forall 0 \leq \tau \leq 1.
\]
Substituting the obtained inequalities into \eqref{divu_triangle} there is an absolute constant $C_4 > 0$ such that 
\[
 ||\operatorname{div} u(t) ||_{s-1} \leq C_1 ||\operatorname{div} u_0||_{s-1} + C_4 \int_0^t ||\operatorname{div} u(\tau)||_{s-1} \,d\tau,\; \forall 0 \leq t \leq 1.
\]
By Gronwall's inequality we then have for any $0 \leq t \leq 1$
\[
 ||\operatorname{div} u(t)||_{s-1} \leq C_1 ||\operatorname{div} u_0||_{s-1} (1+e^{C_4 t}).
\] 
By choosing $\tilde U$ as described above and $C=C_1 (1+e^{C_4})$ we get the claim.
\end{proof}

\noindent
Now we can prove Proposition \ref{prop_local_manifold}.

\begin{proof}[Proof of Proposition \ref{prop_local_manifold}]
By Corollary \ref{exp_restricted}, 
\[
 \exp\big(U^s_{\exp} \cap H^s_\sigma(\R^n;\R^n)\big) \subseteq \Ds^s_\mu(\R^n).
\]
The fact that the differential $d_0 \exp$ is the identity map, together with the inverse function theorem implies that there exists a neighborhood $U' \subseteq U^s_{\exp}$ of $0$ so that 
\[
 \exp:U' \to \Ds^s(\R^n)
\]
is a diffeomorphism onto its image. In particular
\[
 \exp:U' \cap H^s_\sigma(\R^n;\R^n) \to \Ds^s_\mu(\R^n)
\]
is $1-1$. It remains to show that there exists a neighborhood $\tilde U \subseteq U'$ of $0$ so that $\exp(\tilde U) \cap \Ds^s_\mu(\R^n)$ is contained in $\exp\big(\tilde U \cap H^s_\sigma(\R^n;\R^n)\big)$. Arguing by contraposition we show that there exists a neighborhood $\tilde U$ so that any $u_0 \in \tilde U$ with $\exp(u_0) \not\in \Ds^s_\mu(\R^n)$ is an element in $H^s(\R^n;\R^n)\setminus H^s_\sigma(\R^n;\R^n)$. By the formula \eqref{det_formula}, the condition $\exp(u_0) \not\in \Ds^s_\mu(\R^n)$, $u_0 \in U^s_{\exp}$, means for the corresponding solution $u(t)\equiv u(t;u_0)$ and the corresponding flow $\varphi(t)\equiv \varphi(t;u_0)$
\begin{equation}\label{verify_condition}
 \int_0^1 (\operatorname{div} u) (t,\varphi(t,x)) \,dt \neq 0 \mbox{ for some } x \in \R^n.
\end{equation}
In a first step we want to express $\int_0^1 (\operatorname{div} u(t)) \circ \varphi(t)\,dt$ in a convenient way. Integrating \eqref{t_derivative_explicit} gives
\[
 (\operatorname{div} u(t)) \circ \varphi(t) = \operatorname{div} u_0 + \int_0^t R_{\varphi(\tau)} \left(\chi(D)\big(2 (u(\tau) \cdot \nabla) \operatorname{div} u(\tau) + (\operatorname{div} u(\tau))^2 \big) \right) \,d\tau.
\] 
Integrating again we arrive at
\begin{multline}
 \label{integral_div}
 \int_0^1 (\operatorname{div} u(t)) \circ \varphi(t) \,dt = \operatorname{div} u_0 \\
 + \int_0^1 \int_0^t R_{\varphi(\tau)} \left(\chi(D)\big(2 (u(\tau) \cdot \nabla) \operatorname{div} u(\tau) + (\operatorname{div} u(\tau))^2 \big) \right) \,d\tau dt.
\end{multline}
The aim is to bound the $H^{s-1}$-norm of the left-hand side of the latter identity away from $0$. By Lemma \ref{lemma_linear_growth} there exists a ball $\tilde U \subseteq U'$, with $U'$ as above, centered at $0$ and $C_1 > 0$ such that for any $f \in H^{s-1}(\R^n;\R^n)$
\begin{equation}\label{composition_linear}
 ||R_\psi f||_{s-1} \leq C_1 ||f||_{s-1},\quad \forall\psi \in \exp(\tilde U).
\end{equation}
Thus we get for any $u_0 \in \tilde U$
\begin{multline*}
\big|\big| \int_0^1 \int_0^t R_{\varphi(\tau)} \left(\chi(D)\big(2 (u(\tau) \cdot \nabla)\operatorname{div} u(\tau) + (\operatorname{div} u(\tau))^2 \big) \right) \,d\tau dt \big|\big|_{s-1} \\ 
\leq \int_0^1 \int_0^t ||R_{\varphi(\tau)} \left(\chi(D)\big(2 (u(\tau) \cdot \nabla)\operatorname{div} u(\tau) + (\operatorname{div} u(\tau))^2 \big) \right)||_{s-1} \,d\tau dt\\
\leq C_1 \int_0^1 \int_0^t ||\chi(D)\big(2 (u(\tau) \cdot \nabla)\operatorname{div} u(\tau) + (\operatorname{div} u(\tau))^2 \big)||_{s-1} \,d\tau dt
\end{multline*}
where in the last inequality we used \eqref{composition_linear}. By \eqref{chi_smoothing} there is an absolute constant $C_2>0$ such that for any $0 \leq \tau \leq 1$
\begin{multline*}
 ||\chi(D)\big(2 (u(\tau) \cdot \nabla)\operatorname{div} u(\tau) + (\operatorname{div} u(\tau))^2 \big)||_{s-1} \\
 \leq C_2 \big(||(u(\tau) \cdot \nabla) \operatorname{div} u(\tau)||_{s-2} + ||(\operatorname{div} u(\tau))^2||_{s-2} \big).
\end{multline*} 
By multiplication properties of Sobolev functions there exists $C_3 >0$ such that for any $0 \leq \tau \leq 1$
\[
 || (\operatorname{div} u(\tau))^2||_{s-2} \leq C_3 ||u(\tau)||_s ||\operatorname{div} u(\tau)||_{s-1}
\]
and
\[
 ||2 (u(\tau) \cdot \nabla)\operatorname{div} u(\tau)||_{s-2} \leq C_3 ||u(\tau)||_s ||\operatorname{div} u(\tau)||_{s-1}.
\]
From the last two inequalities we conclude that there is an absolute constant $C_4 > 0$ such that for any $0 \leq \tau \leq 1$ and for any $u_0 \in \tilde U$
\[
 ||\chi(D)\big(2 (u(\tau) \cdot \nabla)\operatorname{div} u(\tau) + (\operatorname{div} u(\tau))^2 \big)||_{s-1} \leq C_4 ||u(\tau)||_s ||\operatorname{div} u(\tau)||_{s-1}.
\]
By Lemma \ref{lemma_small_u} -- Lemma \ref{lemma_linear_u} and after shrinking $\tilde U$, if necessary, we get for any $0 \leq \tau \leq 1$ and for any $u_0 \in \tilde U$
\[
 ||\chi(D)\big(2 (u(\tau) \cdot \nabla)\operatorname{div} u(\tau) + (\operatorname{div} u(\tau))^2 \big)||_{s-1} \leq \frac{1}{2} ||\operatorname{div} u_0||_{s-1}.
\]
Thus we get from \eqref{integral_div} for any $u_0 \in \tilde U$
\begin{equation}\label{div_ineq}
 ||\int_0^1 (\operatorname{div} u(t)) \circ \varphi(t) \,dt||_{s-1} \geq \frac{1}{2} ||\operatorname{div} u_0||_{s-1}.
\end{equation}
In particular we see from \eqref{div_ineq}, that for any $u_0 \in \tilde U$ with $\operatorname{div} u_0 \neq 0$ the statement \eqref{verify_condition} holds.
\end{proof}

\noindent
Now we can prove Theorem \ref{th_submanifold}

\begin{proof}[Proof of Theorem \ref{th_submanifold}]
It is to show that the property described in Definition \ref{def_submanifold} holds for $\Ds^s_\mu(\R^n)$. Let $\tilde U \subseteq H^s(\R^n;\R^n)$ be as in the statement of Proposition \ref{prop_local_manifold}. Then
\[
 \exp\big(\tilde U \cap H^s_\sigma(\R^n;\R^n)\big) = \exp(\tilde U) \cap \Ds^s_\mu(\R^n)
\]
and hence $\exp(\tilde U) \cap \Ds^s_\mu(\R^n)$ is a submanifold of $\exp(\tilde U)$ with
\[
 \left. \exp \right|_{\tilde U \cap H^s_\sigma(\R^n;\R^n)}
\]
being a parametrization. To show the analog conclusion for an arbitrary $\psi \in \Ds^s_\mu(\R^n)$ instead of $\operatorname{id} \in \Ds^s_\mu(\R^n)$ we use the group structure of $\Ds^s(\R^n)$ and $\Ds^s_\mu(\R^n)$. We claim that for any $\psi \in \Ds^s(\R^n)$
\[
 R_\psi:\Ds^s(\R^n) \to \Ds^s(\R^n), \quad \varphi \mapsto \varphi \circ \psi
\]
is real analytic. Indeed using the identification of $\Ds^s(\R^n)$ with $\Ds^s(\R^n) - \operatorname{id} \subseteq H^s(\R^n;\R^n)$, one has with $g=\psi - \operatorname{id}$,
\[
 R_\psi:f \mapsto g + f \circ \psi
\]
which is affine and hence real analytic. The map $R_\psi$ is invertible with inverse $R_\psi^{-1}$. Now as $\Ds^s_\mu(\R^n) \subseteq \Ds^s(\R^n)$ is a subgroup one has for any $\psi \in \Ds^s(\R^n)$
\[
 R_\psi\big(\exp(\tilde U \cap H^s_\sigma(\R^n;\R^n))\big) = R_\psi\big(\exp(\tilde U)\big) \cap \Ds^s_\mu(\R^n).
\]
Note that $R_\psi\big(\exp(\tilde U)\big)$ is a neighborhood of $\psi$ in $\Ds^s(\R^n)$. Hence 
\[
 \left. R_\psi \circ \exp \right|_{\tilde U \cap H^s_\sigma(\R^n;\R^n)}
\]
is a real analytic parametrization of $R_\psi\big(\exp(\tilde U)\big) \cap \Ds^s_\mu(\R^n)$. As $\psi \in \Ds^s_\mu(\R^n)$ is arbitrary we get by Definition \ref{def_submanifold} that $\Ds^s_\mu(\R^n)$ is a real analytic submanifold of $\Ds^s(\R^n)$. 
\end{proof}

By Theorem \ref{th_submanifold} we get a differential structure for $\Ds^s_\mu(\R^n)$. An immediate corollary is the following one.

\begin{Coro}\label{coro_exp_restricted}
The exponential map restricts to a real analytic map
\[
 \exp:U^s_{\exp} \cap H^s_\sigma(\R^n;\R^n) \to \Ds^s_\mu(\R^n).
\]
Moreover it is a diffeomorphism around $0$.
\end{Coro}

\begin{Rem}
The tangent space of $\Ds^s_\mu(\R^n)$ at $\operatorname{id} \in \Ds^s_\mu(\R^n)$, as a subspace of $T_{\operatorname{id}}\Ds^s(\R^n)\equiv H^s(\R^n;\R^n)$, is given by
\[
T_{\operatorname{id}}\Ds^s_\mu(\R^n) = H^s_\sigma(\R^n;\R^n).
\]
Indeed the tangent space at $\operatorname{id} \in \Ds^s_\mu(\R^n)$ is by Corollary \ref{coro_exp_restricted} spanned by the vectors
\[
 \left. \partial_\varepsilon \right|_{\varepsilon=0} \exp(\varepsilon v)=v
\]
for $v \in H^s_\sigma(\R^n;\R^n)$. The tangent space at an arbitrary $\psi \in \Ds^s_\mu(\R^n)$ is the right translate of $H^s_\sigma(\R^n;\R^n)$ by $\psi$, i.e. $\tilde v$ is in $T_\psi \Ds^s_\mu(\R^n)$ iff it is of the form
\[
 \tilde v = v \circ \psi
\]
for some $v \in H^s_\sigma(\R^n;\R^n)$.
\end{Rem}
\clearpage
\appendix

\section{Analyticity in real Banach spaces}\label{appendix_analyticity}

The references for this section are \cite{mujica, analyticity}. For differential calculus in Banach spaces see e.g. \cite{dieudonne}. In the following $X$, $Y$, $Z$ will denote {\em real} Banach spaces with the corresponding norms $||\cdot||_X$, $||\cdot||_Y$, $||\cdot||_Z$. We denote by $L^k(X;Y)$ the space of continuous $k$-linear forms on $X \times \ldots \times X$ ($k$-times) with values in $Y$. For any symmetric $\tilde Q \in L^k(X;Y)$ denote by $Q$ the restriction of $\tilde Q$ onto the diagonal. $Q$ is referred to as the homogeneous polynomial associated to $\tilde Q$. For a sequence of symmetric $k$-linear forms $(\tilde Q_k)_{k \geq 0}$, $\tilde Q_k \in L^k(X;Y)$, with the corresponding homogeneouos polynomials $(Q_k)_{k \geq 0}$ consider the power series around $x_0 \in X$
\begin{equation}\label{formal_power}
f(x) = \sum_{k \geq 0} Q_k (x-x_0):=\sum_{k \geq 0} \tilde Q_k(x-x_0,\ldots,x-x_0).
\end{equation}
Following \cite{mujica, analyticity} we call the convergence radius of the power series
\[
 \sum_{k \geq 0} ||Q_k|| t^k, \quad t \in \R
\]
the radius (of convergence) of the series given in \eqref{formal_power}, where we denote by $||Q_k||$ the norm of the homogeneous polynomial $Q_k$, i.e.
\begin{equation}\label{radius_condition}
 ||Q_k|| := \sup_{||x||_X \leq 1} ||Q_k (x)||_Y.
\end{equation}
Thus by the Cauchy-Hadamard formula (see e.g \cite{mujica}) the radius $R$ of the series \eqref{formal_power} is given by
\begin{equation}\label{hadamard_formula}
 1/R = \limsup_{k \to \infty} ||Q_k||^{1/k}.
\end{equation}
We will use this in the following form: If the power series \eqref{formal_power} has radius $R > 0$ we then have
\begin{equation}\label{alternative_description}
\sup_{k \geq 0} ||Q_k|| r^k < \infty
\end{equation}
for any $0 \leq r < R$. On the other hand, if \eqref{alternative_description} holds for any $0 \leq r < R$ then the power series has (at least) radius $R$.\\
Now to the notion of real analyticity.

\begin{Def}\label{def_analytic}
We say that $f:U \subseteq X \to Y$ is real analytic in the open set $U$ if for all $x_0 \in U$ the map $f$ can be represented in a ball around $x_0$ of radius $r > 0$ as a power series of the form \eqref{formal_power} with radius $R \geq r$, i.e. we have
\[
 f(x) = \sum_{k \geq 0} Q_k (x-x_0), \quad ||x-x_0||_X < r.
\]
\end{Def}

As is shown in \cite{analyticity} a power series of the form \eqref{formal_power} with radius $R > 0$ defines a real analytic map in the ball $||x-x_0||_X < R$. There it is also shown that a real analytic map is $C^\infty$ and that composition of real analytic maps is again real analytic. These properties allow the notion of submanifold and the corresponding notion of real analytic maps in the category of real analytic objects. We will use the following form of the definition of a submanifold.

\begin{Def}\label{def_submanifold}
Let $X$ be a real Hilbert space and $U \subseteq X$ a non-empty open subset. We say that $M \subseteq U$,$M \neq \emptyset$, is a real analytic submanifold of $U$ if there is some closed subspace $V \subseteq X$ such that for all $m \in M$ there is some neighborhood $W \subseteq X$ of $m$, an open neighborhood $G$ of $0$ and a real analytic diffeomorphism $\Phi$ ($\Phi^{-1}$ is also real analytic) 
\[
 \Phi:G \to W
\]
such that we have
\[
 \Phi(G \cap V) = W \cap M.
\]
One calls $\left. \Phi \right|_{G \cap V}$ a parametrization of $W \cap M$.
\end{Def}

An existence and uniqueness theorem for analytic ODE's can be found in \cite{dieudonne}. Actually in \cite{dieudonne} they just discuss the situation for complex Banach spaces. But by complexification one immediately gets the analog result for real Banach spaces which reads as

\begin{Prop}\label{prop_analytic_flow}
Let $V:O \subseteq X \to X$ be a real analytic map (vector field) on the open set $O$. For every $w \in O$ there is a $T >0$ and $\delta > 0$ with $B_\delta(w) \subseteq O$ such that for any $u_0 \in B_\delta(w)$ the initial value problem
\begin{equation}\label{analytic_IVP}
 \dot \gamma(t) = V\big(\gamma(t)\big); \quad \gamma(0)=u_0
\end{equation}
has a unique solution $\gamma(t)=\Psi(t,u_0)$. Moreover the flow
\[
 \Psi : (-T,T) \times B_\delta(u_0) \to O
\]
is real analytic.
\end{Prop}

The following criterion (see also \cite{alex2} for a more general result) is used in Lemma \ref{lemma_composition1} to prove that a given map is real analytic.

\begin{Prop}\label{prop_weak_analytic}
Let $\big(X,\langle \cdot,\cdot \rangle_X\big)$ and $\big(Y,\langle \cdot, \cdot \rangle_Y\big)$ be real Hilbert spaces. Let $\phi:U \subseteq X \to Y$ be a map on the open subset $U \subseteq X$. Assume that the following property holds for some $x_0 \in U$ and $R>0$: For any $y \in Y$ we have that the map 
\[
 \langle \phi(\cdot),y \rangle_Y : U \to \R
\]
admits a power series representation with radius $R$ around $x_0$. Then $\phi:U \to Y$ admits a power series representation of radius $R$ around $x_0$.
\end{Prop}

\begin{Rem}\label{rem_weakly_analytic}
This is somehow the version of ''weakly holomorphic implies holomorphic'' suitable for real Hilbert spaces. For complex Hilbert spaces the situation is much easier (see e.g. \cite{mujica}).
\end{Rem}

To prove this proposition we need the following lemma (see also \cite{alex} for a more general formulation). It is just Proposition \ref{prop_weak_analytic} for the case $X=\R$.

\begin{Lemma}\label{lemma_weak_curve}
Let $Y$ be as in Proposition \ref{prop_weak_analytic} and $\gamma:(-R,R) \to Y$ a curve such that for every $y \in Y$ the map
\[
 \langle \gamma(\cdot),y \rangle_Y :(-R,R) \to \R
\]
has a convergent power series with radius $R$ around $0$. Then $\gamma:(-R,R) \to Y$ admits a power series representation
\[
 \gamma(t) = \sum_{k \geq 0} \frac{1}{k!} a_k t^k, \quad(a_k)_{k \geq 0} \subseteq Y
\]
with radius $R$.
\end{Lemma}

\begin{proof}
For a function $f:(-R,R) \to \R$ we define for $k \geq 0$ the finite differences recursively by
\[
 \Delta_h f(t) = f(t+h)-f(t), \ldots, \Delta_h^{k+1} f(t) = \Delta_h^k f(t+h) - \Delta_h^k f(t).
\]
For a fixed $k \geq 0$ these expressions make sense for $t \in (-R,R)$ and $h$ small enough. These finite differences are defined in the same way for $Y$-valued $f$. Furthermore we have for smooth $f$
\begin{equation}\label{derivatives}
  \partial_t^k f(t) = \lim_{h \to 0} \frac{\Delta_h^k f(t)}{h^k}.
\end{equation}
By assumption we have for every  $y \in Y$
\[
 f^{(y)}(t):= \langle \phi(t),y \rangle = \sum_{k \geq 0} \frac{1}{k!} a_k^{(y)} t^k, \quad (a_k^{(y)})_{k \geq 0} \subseteq \R
\]
a power series expansion with radius $R$. By linearity we have
\[
 \frac{\Delta_h^k f^{(y)}(0)}{h^k} = \langle \frac{\Delta_h^k \phi(0)}{h^k},y \rangle_Y.
\]
From \eqref{derivatives} we get 
\[
 \frac{\Delta_h^k f^{(y)}(0)}{h^k} \to a_k^{(y)}
\]
as $h \to 0$. As this holds for every $y \in Y$, we get for some $a_k \in Y$
\[
 \frac{\Delta_h^k \phi(0)}{h^k} \rightharpoonup a_k
\]
i.e., it converges weakly to $a_k$ in $Y$. This $a_k$ has the property
\[
 a_k^{(y)} = \langle a_k,y \rangle_Y
\]
for all $y \in Y$. As $f^{(y)}(t)$ has convergence radius $R$ we have
\[
 \sup_{k \geq 0} \frac{1}{k!} |a_k^{(y)}| r^k < \infty \quad \mbox{ or equivalently } \quad \sup_{k \geq 0} \frac{1}{k!} |\langle a_k,y \rangle_Y| r^k < \infty 
\]
for all $y \in Y$ and $r < R$. By the uniform boundedness principle we then have
\[
 \sup_{k \geq 0} \frac{1}{k!} ||a_k||_Y r^k < \infty
\]
for all $r < R$. This means that $\tilde \phi:(-R,R) \to Y$ defined by
\[
 \tilde \phi(t) := \sum_{k \geq 0} \frac{1}{k!} a_k t^k
\]
is a power series with radius $R$. We have for any $y \in Y$ and $t \in (-R,R)$
\[
 \langle \tilde \phi(t),y \rangle_Y = \sum_{k \geq 0} \frac{1}{k!} \langle a_k, y \rangle_Y t^k = \langle \phi(t), y \rangle_Y
\]
which means $\tilde \phi(t) = \phi(t)$. This shows the lemma.
\end{proof}

\noindent
With the help of this lemma we can prove the proposition.

\begin{proof}[Proof of Proposition \ref{prop_weak_analytic}]
Without loss of generality we assume $x_0=0$. Let $v \in X \setminus \{0\}$. Consider the curve $t \mapsto \phi(t v)$. By assumption $\langle \phi(tv),y \rangle_Y$ has a convergent power series around $0$ with radius $R/||v||_X$. As $R$ does not depend on $y$ we can apply Lemma \ref{lemma_weak_curve} to $t \mapsto \phi(tv)$ and we get that it is a smooth curve. In particular
\[
 Q_k(v):= \left. \partial_t^k \right|_{t=0} \phi(tv)
\]
is well-defined. On the other hand we have by assumption, for any fixed $y \in Y$,
\[
 \langle \phi(v),y \rangle_Y = \sum_{k \geq 0} \frac{1}{k!} Q_k^{(y)}(v)
\]
for some $\R$-valued homogeneuos polynomial $Q_k^{(y)}$ of order $k$, $k \geq 0$. As we have
\[
 \left. \partial_t^k \right|_{t=0} \langle \phi(tv),y\rangle_Y = Q_k^{(y)}(v)
\]
we get for all $y \in Y$
\[
 \langle Q_k(v),y \rangle_Y = Q_k^{(y)}(v).
\]
By \cite{mujica} we know that a weakly continuouos polynomial is a continuous polynomial, i.e. $Q_k(v)$ is a homogeneuous $Y$-valued polynomial in $X$ of order $k$. As the power series with $Q_k^{(y)}$ has radius $R$, we have
\[
 \sup_{k \geq 0} \frac{1}{k!} ||Q_k^{(y)}|| r^k < \infty
\]
for all $y \in Y$ and $r < R$ where $||Q_k^{(y)}||$ is the norm of the $\R$-valued homogeneous polynomial $Q_k^{(y)}$. Again by the uniform boundedness principle we conclude that
\[
 \sup_{k \geq 0} \frac{1}{k!} ||Q_k|| r^k < \infty
\]
for all $r < R$, where here $||Q_k||$ is the norm of the $Y$-valued homogeneous polynomial $Q_k$. Therefore $\tilde \phi:B_R(0) \to Y$ defined by 
\[
 \tilde \phi(v) := \sum_{k \geq 0} \frac{1}{k!} Q_k(v)
\]
is a power series with radius $R$. Now we have for all $y \in Y$ and for all $v \in B_R(0)$
\[
 \langle \tilde \phi(v),y \rangle_Y = \sum_{k \geq 0} \frac{1}{k!} \langle Q_k(v),y \rangle_Y = \sum_{k \geq 0} \frac{1}{k!} Q_k^{(y)}(v) = \langle \phi(v),y \rangle_Y.
\]
Therefore $\tilde \phi(v) = \phi(v)$. Hence the claim.
\end{proof}

\noindent
Sometimes we have to deal with maps which are linear in one entry, i.e. maps of the form
\[
 \phi:(Y \times O) \subseteq Y \times X \to Z
\]
where $\phi(\cdot,x)$, $x \in O$, is linear in the first entry, i.e. $\phi(\cdot,x) \in L(Y;Z)$. For such maps we have the following lemma

\begin{Lemma}\label{lemma_linear}
Assume that
\[
 \phi : Y \times O \to Z
\]
is real analytic and linear in the first entry and has a power series expansion around $(0,x_0) \in Y \times X$ with radius $R$ where $O \subseteq X$ is open and $x_0 \in X$. Then
\begin{eqnarray*}
 \tilde \phi : O &\to& L(Y;Z) \\
 x &\mapsto& \big( y \mapsto \phi(y,x)\big)
\end{eqnarray*}
has a power series expansion with radius $R$ around $x_0$.
\end{Lemma}

\begin{proof}
Without loss of generality we assume $x_0=0$. We have for any fixed $y \in Y$ with $||y||<R$, by Taylor's theorem the following expansion around $x_0=0$
\begin{equation}\label{desired_expansion}
 \phi(y,x) = \sum_{k \geq 0} \frac{1}{k!} d_{2,(y,0)}^k \phi \, x^k
\end{equation}
where $d_2$ denotes the partial derivative in the second entry, i.e. 
\begin{equation}\label{formula1}
 d_{2,(y,0)}^k \phi \,x^k = d^k_{(y,0)} \phi \,(0,x)^k.
\end{equation}
Here we use for a Banach space $W$ and $w \in W$ the notation $w^k$ for $(w,\ldots,w) \in W \times \cdots \times W$ ($k$-times) and $d^k_p \phi \, w^k$ stands for the $k$'th order differential at the point $p$ evaluated in $w^k$. One has 
\begin{equation}\label{formula2}
d_{2,(y,0)}^k \phi \,x^k = \left. \partial_t^k \right|_{t=0} \phi(y,tx)
\end{equation}
We see from \eqref{formula2} that
\[
 y \to d_{2,(y,0)}^k \phi(y,0) \,x^k
\]
is linear. Recall that we have the canonical isomorphism $L^{k+1}(Y \times X\times \cdots \times X;Z) \simeq L^k(X \times \cdots \times X;L(Y;Z))$. Therefore we can look at $d_{2,(\cdot,0)}^k$ as a polynomial in $X$ with values in $L(Y;Z)$. Thus \eqref{desired_expansion} will be the desired expansion. But we have to estimate the corresponding norms to ensure that it has radius $R$. Take $0 < \delta < R$. For any fixed $y \in Y$ with $||y||_Y < \delta$ we have a power series for $x \mapsto \phi(y,x)$ with radius $R-\delta$ -- see \cite{analyticity}. Thus we have
\[
 \sup_{k \geq 0} \frac{1}{k!} \left(\sup_{||x||_X \leq 1} ||d^k_{(y,0)} \phi \,(0,x)^k||_Z \right)r^k < \infty
\]
for all $r < R-\delta$. By the linearity of $d^k_{(y,0)}$ in $y$ this extends to all $y \in Y$. Hence by the uniform boundedness principle 
\[
 \sup_{k \geq 0} \frac{1}{k!} \left(\sup_{||y||_Y \leq 1} \sup_{||x||_X \leq 1} ||d^k_{(y,0)} \phi \,(0,x)^k||_Z \right) r^k < \infty
\]
for all $r < R - \delta$. Thus the expansion \eqref{desired_expansion} has radius $R-\delta$. By letting $\delta \to 0$ we get the claim. 
\end{proof}

Finally we give an example of a real analytic operation which will be needed.

\begin{Lemma}\label{analytic_det}
Let $s > n/2+1$. The map
\[
 \phi:H^{s-1}(\R^n) \times \Ds^s(\R^n) \to H^{s-1}(\R^n),\quad (f,\varphi) \mapsto \frac{f}{\det(d\varphi)}
\]
is real analytic.
\end{Lemma}

\begin{proof}
Consider the map
\[
 \Psi:\Ds^s(\R^n) \to \mathcal L\big(H^{s-1}(\R^n),H^{s-1}(\R^n)\big),\quad \varphi \mapsto \left[ f \mapsto f \cdot \det(d\varphi)\right]
\]
which is real analytic. From \cite{composition} we know that $\phi$ is welldefined (and continuous). This means that $\Psi$ maps into the invertible linear maps. Since the inversion map 
$\operatorname{inv}:T \mapsto T^{-1}$ is real analytic (cf. Neumann series), we see that
\[
 \phi(f,\varphi)=\operatorname{inv}(\Psi) (f)
\]
is real analytic.
\end{proof}

\bibliographystyle{plain}

\end{document}